\documentclass{amsart}
\usepackage{newpxtext}

\usepackage{graphicx}
\usepackage{latexsym}
\usepackage{amsfonts}
\usepackage[all]{xy}
\usepackage{amscd,color}
\usepackage{amsmath}
\usepackage{amssymb}
\usepackage{mathtools}
\usepackage{eucal}
\usepackage{subfig}
\usepackage[pdftex]{hyperref}
\usepackage[T1]{fontenc}
\usepackage[utf8]{inputenc}
\usepackage[english]{babel}
\usepackage{cancel}  % cancelar texto e expressões matemáticas; comandos: \cancelto{value}{expression} ou \xcancel{texto}
\usepackage{mathrsfs}
\usepackage{amsbsy}
\newtheorem*{acknowledgement}{Acknowledgement}
\newtheorem{corollary}{Corollary}

\newtheorem{definition}{Definition}
\newtheorem{lemma}{Lemma}
\newtheorem{proposition}{Proposition}
\newtheorem{remark}{Remark}
\newtheorem{theorem}{Theorem}
\newtheorem{example}{Example}
\numberwithin{equation}{section}

\begin{document}

	\title[Overdetermined problems on Riemannian manifolds]{Semilinear overdetermined problems, a divergence formula, and geometric inequalities}


	%    author one information

	\author{Benedito Leandro}
	\address{B. Leandro - Departamento de Matem\'atica, Universidade de Bras\'ilia\\
		Bras\'ilia-DF, 70910-900, Brazil.}
	\email{bleandrone@mat.unb.br}
	\thanks{B.Leandro was partially supported by CNPq/Brazil Grant PQ 303618/2026-4, CNPq/Brazil Grant UNIVERSAL 400078/2025-2, and FAPDF - 00193-00001678/2024-39.}

	\author{Ilton Menezes}
	\address{I. Menezes - Centro das Ci\^encias Exatas e das Tecnologias, Universidade Federal do Oeste da Bahia, CEP 47808-021, Barreiras, BA, Brazil.}
	\email{ilton.menezes@ufob.edu.br}

	\author{Rafael Novais}
	\address{R. Novais - Instituto Federal de Educa\c{c}\~ao, Ci\^encia e Tecnologia, Campus Posse, CEP 73900-000, Posse, GO, Brazil.}
	\email{rafael.novais@ifgoiano.edu.br}

	\thanks{Corresponding author: B. Leandro (bleandrone@mat.unb.br)}


	\date{\today}

	\begin{abstract}
		Considering an $n$-dimensional compact Riemannian manifold with a boundary that satisfies a Serrin-type problem, we prove sharp upper and lower bounds for the area of such a boundary. Then, we present the main result of this paper: a divergence formula for a special vector field on a given Riemannian manifold. This divergence formula is closely related to sub-static manifolds and related metrics, e.g., $V$-static, static, and electrostatic manifolds. We show Minkowski-type inequalities for $V$-sub-static manifolds under different Neumann boundary conditions.  Moreover, we prove an area-charge inequality for compact (and noncompact) electrostatic manifolds. The main application of the divergence formula presented in this work generalizes the classical result of Boucher--Gibbons--Horowitz: we prove that an asymptotically hyperbolic space of a static spacetime satisfying the null convergence condition must be the hyperbolic space.
	\end{abstract}

	\maketitle

	\tableofcontents

	\section{Introduction}

	Rigidity phenomena in geometric analysis are frequently studied via obstruction results. For instance, geometric inequalities can yield full classifications by ruling out potential new examples. Inequalities involving the mass and the area of the boundary of static spaces were fundamental for the rigidity result proved by Robinson \cite{robinson} using a distinguished and inspired divergence formula that generalized the work of Israel \cite{israel}. In this context, an important result of Boucher, Gibbons, and Horowitz \cite{boucher1} asserts that the de Sitter system maximizes the boundary area among all compact vacuum static spaces with constant positive scalar curvature and a connected boundary. More precisely, they proved the following:
	\begin{quote}
		{\it Consider a three-dimensional, compact, oriented, static space with a connected boundary and scalar curvature equal to \(6\). Then the area of its boundary \(\partial M\) satisfies the inequality
			\begin{equation*}
				\operatorname{Area}(\partial M) \leq 4\pi.
			\end{equation*}
			Moreover, equality holds if and only if \((M^3, g)\) is isometric to the de Sitter system, i.e., the standard hemisphere.}
	\end{quote} In recent years, also motivated by the classical isoperimetric inequality, the study of boundary and volume estimates for special classes of manifolds (or metrics) has made significant advances. Such estimates have been established for static spaces, e.g., in \cite{agostiniani2017,ambrozio,cederbaum,chrusciel,tiarlos,freitas,HMR,sh}, as well as for critical metrics of the volume functional, e.g., in \cite{corvino}.

	In this spirit, by employing the recent generalized Reilly formula due to Qiu and Xia \cite{qiu} (see also \cite{li}), we establish a sharp lower boundary estimate for a compact manifold with Ricci curvature bounded from below, satisfying a semilinear overdetermined problem (i.e., a Serrin-type problem). In the 1980s, Reilly \cite{reilly} used an inspired formula to prove the following Minkowski inequality for compact
	Riemannian manifolds with nonnegative Ricci curvature and convex boundary. Reilly's proof is based on the solvability of the following Neumann problem
	\begin{equation*}
		\left\{
		\begin{array}{rcll}
			\Delta f &=& 1 ,\quad\text{in}\quad M,\\\\
			f &=& c,\quad\text{on}\quad \partial M,
		\end{array}
		\right.
	\end{equation*}
	He applied his famous formula to the solution of the above system to obtain
	\begin{equation*}
		\left(\int_{\partial M}da\right)^{2}\geq n\left(\int_M dv\right)\left(\int_{\partial M}\dfrac{H}{n-1}\,da\right),
	\end{equation*}
	where equality holds if and only if $M$ is a ball. In this paper, we generalize the Minkowski inequality under different Neumann boundary conditions and weaker Ricci curvature assumptions.

	The works we referenced earlier demonstrate that semilinear overdetermined problems on manifolds, including static and $V$-static manifolds, are crucial for proving geometric inequalities related to the area of the boundary of these spaces. We will now define a class of semilinear overdetermined systems that arise in several contexts of Einstein-type manifolds \cite{leandro}. The following system can be seen as a generalization of the Serrin-type problem; see \cite{farina,Weinberger} and the references therein.
	\begin{definition}\label{defserrin}
		Let $(M^n,g)$ be an $n$-dimensional compact Riemannian manifold with boundary $\partial M$, and let $f\in C^\infty(M)$ satisfy $f>0$ in $\operatorname{int}(M)$ and the generalized Serrin boundary conditions, i.e.,
		\begin{equation}\label{Serrin}
			\left\{
			\begin{array}{rcll}
				\Delta f&=&-\lambda f-\beta,\quad\text{in}\quad M,\\\\
				f&=&0,\quad\text{on}\quad \partial M,\\\\
				\dfrac{\partial f}{\partial\nu}&=&-c,\quad\text{on}\quad \partial M,
			\end{array}
			\right.
		\end{equation}
		where $\Delta$ and $\nu$ stand for the Laplacian on $M$ and the normal vector field along $\partial M$, respectively, with respect to the metric $g$. Here, $\lambda$ and $\beta$ are constants, and $c>0$ is a constant.
	\end{definition}

	We will consider different generalizations of Definition~\ref{defserrin} throughout the manuscript.
	Serrin \cite{serrin} proved that if $M$ is a bounded domain in $\mathbb{R}^n$ that admits a solution to \eqref{Serrin} with $\lambda=0$ and $\beta=1$, then $M$ is a ball and
	$f$ is radially symmetric. More precisely, the solution is equivalent to
	\begin{align*}\label{serrinball}
		f(x)=\dfrac{1-|x|^2}{2n},\quad\text{in}\quad M=B_1(0).
	\end{align*}
	We will refer to $(B_1(0),f)$ as the Serrin solution. In fact, solutions for Serrin's problem \eqref{Serrin} with $\beta=1$ are bounded domains in $\mathbb{H}^n$, $\mathbb{R}^n$, and $\mathbb{S}_{+}^n$ when $\lambda=-n$, $\lambda=0$, and $\lambda=n$, respectively. We will provide more important examples for Definition~\ref{defserrin} in Section~\ref{back}.

\subsection{Notation}	Let us establish some notation before introducing our results. We denote the area of the boundary $\partial M$ by $\operatorname{Area}(\partial M)$, the volume of $M$ by $\operatorname{Vol}(M)$, and the Ricci curvature of a Riemannian manifold $M^n$ with metric $g$ by $\operatorname{Ric}$. Moreover, $R^{\partial M}$, $H$, and $A$ stand for the scalar curvature, mean curvature, and second fundamental form of $\partial M$, where the normal vector field is $\nu$. Here, $R$ corresponds to the scalar curvature of $M$ with respect to the metric $g$, $\mathcal{X}(\partial M)=2(1-G(\partial M))$ stands for the Euler characteristic of $\partial M$, $G(\partial M)$ is the genus of $\partial M$, and $\omega_{n-1}$ is the area of the standard $(n-1)$-sphere. In the results presented in this manuscript, we assume $n\ge3$ and that $\partial M$ is connected (even though most of the results also hold for a disconnected boundary).

	\subsection{Bounds for the Area of the Boundary}
	We begin by providing a lower bound on the area of the boundary of a compact manifold with pinched Ricci curvature. This result was inspired by Reilly's integral formula and a divergence formula for the Ricci curvature. Similar results for static metrics can be found in \cite[Theorem D]{ambrozio}.

	\begin{theorem}\label{theoBetazero}
		Let $(M^n,g)$ be a compact Riemannian manifold with mean-convex boundary satisfying $\operatorname{Ric}\geq\dfrac{(n-1)\lambda}{n}g$ and \eqref{Serrin} with $\beta\leq0$. Then,
		\begin{equation*}
			\int_{\partial M}R^{\partial M}\,da\leq\dfrac{(n-2)R_0}{n}\,\operatorname{Area}(\partial M),
		\end{equation*}
		$\beta=0$, $\lambda>0$, and $\partial M$ is a minimal surface. Here, $R_0=\max_M R$.
		For the three-dimensional case,
		\begin{equation*}
			12\pi\mathcal{X}(\partial M)\leq R_0\operatorname{Area}(\partial M).
		\end{equation*}
		Equality holds if and only if $(M,g)$ is the standard hemisphere (Example~\ref{ex1}).
	\end{theorem}

	\begin{remark}
		The standard hemisphere (Example~\ref{ex1}) and the cylinder (Example~\ref{example2}) are among the most important examples in this work of compact manifolds with boundary satisfying Definition~\ref{defserrin} with $\beta=0$; see Section~\ref{back} for more details.
	\end{remark}

	Reilly's formula has been shown to be a promising tool for gaining new geometric inequalities and obstruction results (cf.~\cite{freitas}). In the following theorem, we revisit Reilly's formula provided by \cite{qiu} to obtain a volume estimate for the Serrin-type problem \eqref{Serrin}. The result states how the constant $\lambda$ of Definition~\ref{defserrin} will influence the relationship between the volume of the manifold and the area of the boundary.

	\begin{theorem}\label{prop11}
		Let $(M^n,g)$ be a compact Riemannian manifold satisfying \eqref{Serrin} with
		\begin{equation*}
			\operatorname{Ric}\geq\dfrac{(n-1)\lambda}{n}g,
		\end{equation*}
        where $c>0$, $\lambda\geq0$, and $\beta\geq0$ are constants.
		Then,
		\begin{equation*}
			\operatorname{Area}(\partial M)\leq\sqrt{\dfrac{2n\lambda}{n+2}+\dfrac{\beta^2}{c^2}}\,\operatorname{Vol}(M),
		\end{equation*}
		 Assuming that $f$ and $g$ extend smoothly to $\partial M$, equality holds if and only if $(M^n,g,f)$ is a compact geodesic ball $B_{r_0}(p)$ and $f(r)=\dfrac{\beta}{2n}(r_0^2-r^2)$, where $r=d(p,\cdot)$ is the distance function from a point $p\in M$ satisfying $\nabla f(p)=0$.
		Moreover, the metric $g$ admits the global warped product decomposition
		\begin{equation*}
			g=dr^2+\psi(r)^2g_{\mathbb{S}^{n-1}}.
		\end{equation*}
		Here, $\psi(r)=\dfrac{\beta}{n}r$, and $r_0 =d(p,q)$, where $q\in\partial M$.
	\end{theorem}

	Notice that Theorem~\ref{theoBetazero} and Theorem~\ref{prop11} can be seen as obstruction results, in terms of the area of the boundary, for the existence of solutions to the Serrin-type problem \eqref{Serrin} on a given compact manifold $M^n$ with boundary $\partial M$.

	\subsection{Divergence Formula and Applications}

	Inspired by the famous divergence formula of Robinson \cite{robinson} and Reilly's integral formula \cite{li,qiu,reilly}, we will prove a divergence formula for an arbitrary Riemannian manifold. This divergence formula can be applied in the analysis of several Einstein-type manifolds \cite{leandro}.

	Let $(M^n,g)$ be an $n$-dimensional compact oriented manifold, and let $\mathcal{M}$ be the set of smooth Riemannian structures on $M$ of volume 1. The total scalar curvature map $\mathcal{S}:\mathcal{M}\to\mathbb{R}$ is defined by
	\begin{equation*}
		\mathcal{S}(g) =\int_MR_g\,dv_g
	\end{equation*}
	Let $h$ be a symmetric $2$-tensor on a given Riemannian manifold $M$ with metric tensor $g$. The linearization $\mathfrak{L}_g(h):=DR_g(h)$ of the scalar curvature on $(M,g)$ is given by
	\begin{equation*}
		\mathfrak{L}_g(h) =-\Delta_g(\operatorname{tr}_g h)+\operatorname{div}_g\operatorname{div}_g h-g(h,\operatorname{Ric}_g),
	\end{equation*}
	and its formal $L^2$-adjoint is
	\begin{equation*}
		\mathfrak{L}_g^{\ast}(f)= -f\operatorname{Ric} + \nabla^2f - (\Delta f)g,
	\end{equation*}
	where $\nabla^2$ stands for the Hessian form on a Riemannian manifold $(M^n,g,f)$, and such a function $f:M\to\mathbb{R}$ is called a potential function (see \cite{Besse,corvino,miaotam2009,miaotam} and the references therein).

	\begin{definition}\label{substatic}
		A Riemannian manifold $(M^n,g,f)$ is called a \textbf{sub-static manifold} if
		\begin{equation*}
			-\mathfrak{L}^{\ast}_g(f)\geq0.
		\end{equation*}
 A Lorentzian manifold $(\widehat{M},\, \widehat{g})$ is said to satisfy the Null
Convergence Condition if \[\widehat{\operatorname{Ric}}(L,\,L)\geq0\] for any null vector field $L$. It is known that the sub-static condition
of a Riemannian triple $(M^n,\,g,\,f)$ is equivalent to the  Null Convergence Condition (NCC) of the corresponding spacetime $\widehat{M}=\mathbb{R}\times M^n$ with metric $\widehat{g}=-f^2dt^2+g$; see \cite[Lemma 3.8]{mutao}.
	\end{definition}

    		In particular,
		\begin{itemize}
			\item[(i)] When $\mathfrak{L}^{\ast}_g(f)= k g$, we call $(M^n,g,f)$ a $V$-static manifold with
			\begin{equation*}
				\Delta f=-\dfrac{R}{(n-1)}f - \dfrac{nk}{(n-1)},\quad k\in\mathbb{R},
			\end{equation*}
			and constant scalar curvature $R.$

			\item[(ii)] When $\mathfrak{L}^{\ast}_g(f)=0$, we call $(M^n,g,f)$ a static manifold with
			\begin{equation*}
				\Delta f=-\dfrac{R}{(n-1)}f,
			\end{equation*}
			and constant scalar curvature $R.$

			\item[(iii)] When
			\begin{equation*}
				\mathfrak{L}^{\ast}_g(f)=-2f(|E|^2g-E^{\flat}\otimes E^{\flat}),
			\end{equation*}
			we have an electrostatic manifold with
			\begin{equation*}
				\Delta f=\dfrac{2}{n-1}((n-2)|E|^2-\lambda)f,
			\end{equation*}
			scalar curvature given by $R=2(|E|^2+\lambda)$, and $d(fE^{\flat})=0$. Here, $E$ denotes the electric field and $E^{\flat}$ its dual $1$-form.
		\end{itemize}
		For more on Einstein-type manifolds, see \cite{ambrozio,corvino,tiarlos,freitas,leandro,miaotam}.

	Now we are ready to prove a new divergence formula for a smooth function on a given Riemannian manifold. This formula is closely related to the linearization of the scalar curvature map.

	The following theorem is a divergence formula based on Kato's inequality for any smooth function on a given Riemannian manifold. This formula is the main ingredient of this paper. As we have noted before, divergence formulas are useful for deriving rigidity results for Einstein-type manifolds, such as static, $V$-static, electrostatic, and quasi-Einstein manifolds. Moreover, divergence formulas for static manifolds have been proved to be related to potential theory, and this combination of theories has been used to prove the rigidity of static manifolds as well as the positive mass theorem; see \cite{agostiniani2017,agostiniani2020,agostiniani2024,cederbaum}.

	\begin{theorem}\label{divkato}
		Let $(M^n,g,f)$ be a Riemannian manifold such that $f:M\to\mathbb{R}$ is a smooth function. Then,
		\begin{multline*}
			f\operatorname{div}\left(\dfrac{1}{2f}\nabla|\nabla f|^2-\dfrac{\Delta f}{nf}\nabla f\right)
			\geq \dfrac{n}{n-1}\dfrac{1}{|\nabla f|^2}\left|\dfrac{1}{2}\nabla|\nabla f|^2-\dfrac{\Delta f}{n}\nabla f\right|^2\\
			-\dfrac{1}{f}\mathfrak{L}_g^{\ast}(f)(\nabla f,\nabla f)
			+\dfrac{n-1}{n}f\left\langle\nabla\left(\dfrac{\Delta f}{f}\right),\nabla f\right\rangle.
		\end{multline*}
		In particular,
		\begin{align*}
			f\operatorname{div}\left(\dfrac{1}{2f}\nabla|\nabla f|^2-\dfrac{\Delta f}{nf}\nabla f\right)
			\geq
			-\dfrac{1}{f}\mathfrak{L}_g^{\ast}(f)(\nabla f,\nabla f)
			+\dfrac{n-1}{n}f\left\langle\nabla\left(\dfrac{\Delta f}{f}\right),\nabla f\right\rangle.
		\end{align*}
		Equality holds if and only if $\mathring{\nabla}^2f =0$.
	\end{theorem}

	\subsubsection{Compact Manifolds}
	First, we will explore the geometric properties of the level sets of the potential function $f$ given by Theorem~\ref{divkato}. The level sets of $f$ also have a physical interpretation and are closely related to static black holes and photon surfaces when we consider static manifolds; see \cite{cederbaum}. In the following results, we will consider two cases separately: $f^{-1}(0)$ (black holes) and $f^{-1}(\kappa)$ (photon sphere), where $\kappa\in(0,+\infty)$. Interestingly, $V$-static metrics also have a deep connection with general relativity. In particular, McCormick \cite{mccormick} proved that $V$-static manifolds arise in the study of asymptotically hyperbolic manifolds as critical points of the volume-renormalized mass.

	Divergence formulas are very useful in the study of static metrics and general relativity \cite{cederbaum}. Moreover, such formulas are related to potential theory (cf. \cite{agostiniani2020}), and this link is evidenced by \cite[Section 7]{cederbaum}, where the authors proved that the uniqueness theorem for the static space provided by \cite{agostiniani2017} via potential theory could be established using a generalized Robinson's divergence formula. Moreover, a combination of divergence formulas and potential theory was used to prove the positive mass theorem; see \cite{agostiniani2024}. In what follows, we present some interesting applications of Theorem~\ref{divkato} to sub-static manifolds. We will start with electrostatic manifolds.

	\begin{theorem}\label{sabo}
		Let $(M^n,g,f,E)$ be a compact electrostatic manifold with $\lambda\geq0$ satisfying
		\begin{equation*}
			\left\{
			\begin{array}{rcll}
				\Delta f&=&\dfrac{2}{n-1}((n-2)|E|^2-\lambda)f,\quad\text{in}\quad M,\\\\
				d(fE^{\flat})&=&0,\quad\text{in}\quad M,\\\\
				f&=&0,\quad\text{on}\quad \partial M,\\\\
				\dfrac{\partial f}{\partial\nu}&=&-c,\quad\text{on}\quad \partial M.
			\end{array}
			\right.
		\end{equation*}
		Then, if $\operatorname{div}(E)=0$ we have
		\begin{multline*}
			c\int_{\partial M}\dfrac{R^{\partial M}}{2}\,da
			\geq\dfrac{c(n-2)\lambda}{n}\operatorname{Area}(\partial M)
			+\dfrac{c\omega_{n-1}^2(n-1)(n-2)}{2\operatorname{Area}(\partial M)}Q(\partial M)^2\\
            +\dfrac{c(n-2)\omega_{n-1}^2(n-1)(n-2)}{n\operatorname{Area}(\partial M)}Q(\partial M)^2
            -\dfrac{4(n-2)}{n}\int_Mf\bigl(|\nabla E|^2+\operatorname{Ric}(E,E)\bigr)\,dv .
		\end{multline*}
		Equality holds if and only if $(M^n,g,f,E)$ is the standard hemisphere with $E=0$ (Example~\ref{ex1}).
	\end{theorem}

	\begin{remark}
		Note that if we assume $n=3$ and \[\dfrac{8c\pi^2Q(\partial M)^2}{\operatorname{Area}(\partial M)}
           \ge \int_Mf\bigl(|\nabla E|^2+\operatorname{Ric}(E,E)\bigr)\,dv\] in the proof of the above theorem, we will obtain
	\begin{equation*}
			4\pi
			\geq\dfrac{\lambda}{3}\operatorname{Area}(\partial M)
			+\dfrac{16\pi^2}{\operatorname{Area}(\partial M)}Q(\partial M)^2.
		\end{equation*}
		This estimate is the same as that obtained by \cite{tiarlos}; see also \cite{freitas}. However, since the cosmological constant $\lambda$, as observed through physical experimentation, is supposed to be very small and positive, it is preferable and more natural to consider a different condition than $|E|^2\leq\lambda$; see more in \cite{freitas} and in the references therein. It is important to note that the condition $|E|^2\geq\lambda$ cannot be assumed. 
	\end{remark}

	The following result is a Minkowski-type inequality for a sub-static manifold, where we used Theorem~\ref{divkato} as an approach; see also \cite[Theorem 1.7]{li}. We present strong rigidity when equality holds.

	\begin{theorem}\label{thm:k-sub-static}
		Let $(M^n,g,f)$ be a compact sub-static Riemannian manifold satisfying
		\begin{equation*}
			\left\{
			\begin{array}{rcll}
				\Delta f&=&-\lambda f-\beta,\quad\text{in}\quad M,\\\\
				f&=&\kappa,\quad\text{on}\quad \partial M,\\\\
				\dfrac{\partial f}{\partial\nu}&=&-c,\quad\text{on}\quad \partial M.
			\end{array}
			\right.
		\end{equation*}
		where $\kappa>0$, $c>0$, and $\beta\geq0$. Then,
		\begin{equation*}
			(\lambda\kappa+\beta)\operatorname{Area}(\partial M)^2\geq n\left(\lambda\int_Mf\,dv + \beta\operatorname{Vol}(M)\right)\int_{\partial M}\dfrac{H}{n-1}\,da.
		\end{equation*}
		Moreover, under the conditions of Theorem \ref{thm:almaraz-barbosa-obata}, equality holds if and only if $(M,g)$ is isometric to a geodesic ball in the standard $n$-sphere with $\beta=0$ (Example~\ref{esferanaesfera}).
	\end{theorem}

	The last theorem indicates that the sub-static condition should be relaxed since equality in Theorem~\ref{thm:k-sub-static} holds for a geodesic ball in the standard sphere with a $V$-static metric; see the first item of Example~\ref{miaomiao}. Thus, we will generalize the above theorem to a more general setting.

	Now, we will see that if we want to relax the boundary conditions in Theorem~\ref{thm:k-sub-static}, we will pay a price; namely, we need to impose an additional geometric restriction on the boundary (the boundary must be convex). We know from standard elliptic PDE theory \cite[System 5.5]{li} that the solution $f$ for the Neumann boundary value problem
	\begin{equation}\label{theo6li}
		\left\{
		\begin{array}{rcll}
			\operatorname{div}\left(f^2\nabla\left(\dfrac{U}{f}\right)\right)&=& f,\quad\text{in}\quad M,\\\\
			f^{2}\dfrac{\partial}{\partial\nu}\left(\dfrac{U}{f}\right)&=& c f,\quad\text{on}\quad \partial M.
		\end{array}
		\right.
	\end{equation}
	exists and is unique from the Fredholm alternative. Here, $U$ is an arbitrary smooth function on $M$. Considering $(M^n,g,f)$ as an $n$-dimensional compact sub-static manifold with boundary $\partial M$, where $f>0$ is a solution to the Neumann problem above given such that
	\begin{align*}
		A(X,Y) - \dfrac{1}{f}\dfrac{\partial f}{\partial\nu}g(X,Y) \geq0,
	\end{align*}
	for any tangent vector fields $X$ and $Y$ of $\partial M$,
	Li and Xia \cite{li} proved the following Minkowski-type inequality:
	\begin{align}\label{minkli}
		\left(\int_{\partial M}f\,da\right)^2
		\geq n\left(\int_Mf\,dv\right)\left(\int_{\partial M}\dfrac{fH}{(n-1)}\,da\right).
	\end{align}
	Their proof is based on a general integral formula (Theorem~\ref{li}) provided by the authors, which also demands that $\dfrac{\nabla^2f}{f}$ is continuous up to the boundary. The next theorem extends the inequality \eqref{minkli} in a broad geometrical sense (i.e., we will consider $V$-sub-static manifolds and a different Neumann problem) and includes a simple application of Theorem~\ref{divkato}.  To that end, we need to provide the following definition:

	\begin{definition}\label{vsubstatic}
		An $n$-dimensional Riemannian manifold $(M^n,g,f,\beta)$  such that
		\begin{align*}
			-\mathfrak{L}_g^{\ast}(f)  +  \dfrac{(n-1)\beta}{n}g \geq0
		\end{align*}
		is called a \textbf{$V$-sub-static manifold}, where $\beta\in\mathbb{R}$ and
		\begin{equation*}
			-\mathfrak{L}_g^{\ast}(f)= f\operatorname{Ric}-\nabla^2f+(\Delta f)g.
		\end{equation*}
	\end{definition}

	Let us prove a rigidity result for compact $V$-sub-static manifolds. This result motivates the next theorem (Theorem~\ref{minkhesssemtraco}), showing that the sign of the $V$-sub-static potential $f$ plays an important role in the study of such manifolds.
	\begin{proposition}\label{sinal da f na v sub static}
		Let $(M^n,g,f,\beta)$ be a compact $V$-sub-static manifold without boundary and constant scalar curvature $R$ satisfying
		\begin{equation*}
			\left\{
			\begin{array}{rcll}
				\Delta f & = & - \dfrac{R}{(n-1)}f - \dfrac{nk}{(n-1)},\quad\text{in}\quad  M,\\\\
				f&>& 0,\quad\text{in}\quad M.
			\end{array}
			\right.
		\end{equation*}
		Then, $f$ must be a constant function. Here, $k\in[0,+\infty)$.
	\end{proposition}

	The following theorem is a generalization of Theorem~\ref{thm:k-sub-static}, but the techniques and geometric tricks applied in the proof are slightly different. For that reason, we will keep both results in the text for the sake of completeness and coherence.  Moreover, it will become clear that Theorem~\ref{thm:k-sub-static} can also be considered under the assumption of $V$-sub-static manifolds.

	\begin{theorem}\label{minkhesssemtraco}
		Let $(M^n,g,f,\beta)$ be a compact $V$-sub-static manifold with boundary $\partial M$ on which $f>0$. Assume that
		\begin{equation}\label{theo6}
			\left\{
			\begin{array}{rcll}
				\Delta f & = & - \lambda f - \beta,\quad\text{in}\quad  M,\\\\
				\dfrac{\partial f}{\partial\nu}&=& -c f,\quad\text{on}\quad \partial M.
			\end{array}
			\right.
		\end{equation}
		and
		\begin{equation*}
			\int_{\partial M}\dfrac{1}{f}(A(\nabla_{\partial M}f,\nabla_{\partial M}f) + c|\nabla_{\partial M}f|^2)\,da \geq 0,
		\end{equation*}
		where $\lambda$ and $\beta$ are constants such that $c:=c(\lambda,\beta)$ is a positive constant.
		Therefore,
		\begin{align*}
			\left(\int_{\partial M}f\,da\right)\left(\lambda\int_{\partial M}f\,da + \beta\operatorname{Area}(\partial M)\right)
			\geq n\left(\lambda\int_Mf\,dv + \beta \operatorname{Vol}(M)\right)\int_{\partial M}\dfrac{fH}{(n-1)}\,da.
		\end{align*}
		Equality holds if and only if $\mathring{\nabla}^2f=0$. In particular, assume that Theorem~\ref{thm:almaraz-barbosa-obata} holds, $\partial M$ is convex, $\lambda=n$, and $f+\beta/n\neq0$ on $\partial M$. Then, equality holds if and only if $(M,g)$ is isometric to a geodesic ball in the standard sphere $\mathbb{S}^n$ (Example~\ref{esferanaesfera}).
	\end{theorem}

	\begin{remark}
		We cannot apply \cite[Theorem~2]{tashiro} in the proof of the rigidity of Theorem~\ref{thm:k-sub-static}, Theorem~\ref{minkhesssemtraco}, and Theorem~\ref{prop11}. Indeed, Tashiro's theorem concerns complete Riemannian manifolds and uses geodesic completeness to ensure that the integral curves of the concircular field are defined for
		all values of their parameters. A compact manifold with a nonempty boundary does not satisfy this global requirement since its geodesics may reach the
		boundary in finite time.
	\end{remark}

	\begin{remark}
		The system \eqref{theo6li} is equivalent to \eqref{theo6} by considering $U=1$, $\lambda=1$, and $\beta=0$. Let us point out that by choosing $\lambda=1$ and $\beta=0$, we obtain \cite[Theorem 1.7]{li} through a different approach; i.e., we can obtain \eqref{minkli} by using Theorem~\ref{divkato}. A very interesting case is when $\lambda = 0$, which generalizes the classical Minkowski inequality proved by Reilly in \cite{reilly}.  Moreover, we do not assume that the scalar curvature $R$ is constant (as we do for $V$-static manifolds).
	\end{remark}

	\subsubsection{Non-Compact Manifolds}

	From now on, we will consider applications of Theorem~\ref{divkato} to noncompact Riemannian manifolds. In fact, we are interested in proving results that establish the rigidity of hyperbolic (anti-de Sitter) space. The model we will use as a reference is the Reissner--Nordstr\"om manifold with a cosmological constant $\lambda$, which is a three-parameter family (characterized by the ADM mass \(\mathfrak M\), the charge \(Q\), and the cosmological constant \(\lambda\)) of a static, electrically charged solution to the Einstein--Maxwell equations. Consider the system \((M^n,g_{_{\mathrm{RNdS}}},f,E)\), where
	\begin{equation*}
		M^n=I\times\mathbb{S}^{n-1}.
	\end{equation*}
	for some interval \( I \subset \mathbb{R} \) determined by the roots of the static potential
	\begin{equation}\label{RNDS}
		f(r)=\sqrt{1-\dfrac{2\mathfrak M}{r^{n-2}}+\dfrac{Q^2}{r^{2(n-2)}}-\dfrac{2\lambda r^2}{n(n-1)}},
	\end{equation}
	and
	\begin{align}\label{E-vector}
		E=\sqrt{\dfrac{(n-1)(n-2)}{2}}\dfrac{Q}{r^{n-1}}f(r)\,\partial_r.
	\end{align}
	The metric is given by
	\begin{equation*}
		g_{\mathrm{RNdS}}=f(r)^{-2}\,dr^2+r^2g_{\mathbb{S}^{n-1}}.
	\end{equation*}
	where \( g_{\mathbb{S}^{n-1}} \) denotes the standard metric on the unit sphere \( \mathbb{S}^{n-1} \) and \( r \) represents the radial coordinate. When the cosmological constant vanishes identically, the space is referred to as the Reissner--Nordstr\"om space. Setting $Q=\lambda=0$, we obtain the Schwarzschild static manifold. Moreover, when $Q=0$, we have a Schwarzschild--de Sitter static manifold if $\lambda>0$ and a Schwarzschild--anti-de Sitter static manifold if $\lambda<0$. Setting $\mathfrak M=Q=0$ and $\lambda<0$, we obtain hyperbolic space. The charge \( Q \) of a hypersurface \( \Sigma \) in \( (M^n, g) \) is given by
	\begin{equation}\label{charge}
		Q(\Sigma)=\dfrac{1}{\omega_{n-1}}\sqrt{\dfrac{2}{(n-1)(n-2)}}\int_{\Sigma}\langle E,\nu\rangle\,da,
	\end{equation}
	where $\omega_{n-1}$ denotes the area of the standard unit $(n-1)$-sphere, and $\nu$ is the unit normal vector to $\Sigma$.

	It is also known that the ADM mass $\mathfrak M$ (see \cite{miao2015}) can be computed
	using the Ricci curvature as follows. Define
	\begin{equation*}
		m(r)=-\dfrac{1}{(n-1)(n-2)\omega_{n-1}}\int_{\mathbb{S}(r)}\left(\operatorname{Ric}-\dfrac{R}{2}g\right)(X,\eta)\,da,
	\end{equation*}
	where $\operatorname{Ric}$ and $R$ are the Ricci tensor and the scalar curvature, respectively. Here, $X$ is the Euclidean conformal Killing vector field
	\begin{equation*}
		x^i\dfrac{\partial}{\partial x^i},
	\end{equation*}
	$\eta$ is the unit outward normal, and $da$ is the area element of $\mathbb{S}(r)$ with respect to $g$.
	If the limit
	\begin{equation*}
		m_l=\lim_{r\to+\infty}m(r)
	\end{equation*}
	exists, then
	\begin{equation*}
		\mathfrak M=m_l.
	\end{equation*}
	The ADM mass and the positive mass theorem \cite{Herzlich,wang}  will play an important rule in the proof of the following theorems.

	We will extend the result of \cite{boucher1}, in which the authors proved a rigidity result for asymptotically anti-de Sitter static manifolds with a negative cosmological constant $\lambda$. More precisely, we will prove that an asymptotically hyperbolic sub-static manifold must be isometric to hyperbolic space. To that end, recall the positive mass theorem for such spaces \cite{Herzlich,wang}; see also \cite{boucher1,gibbons}. In the following theorem, we will consider an asymptotically Schwarzschild--anti-de Sitter ($\lambda<0$) static manifold as a model.

	\begin{theorem}\label{coroasymp}
		Let $(M^n,g,f)$ be an asymptotically Schwarzschild--anti-de Sitter sub-static manifold without boundary such that
		\begin{equation*}
			\Delta f = n f ,\quad\text{in}\quad M,
		\end{equation*}
		Then,
		\begin{equation*}
			\mathfrak M\leq0.
		\end{equation*}
		Consequently, if $(M^n,g,f)$ is a spin manifold, it must be the anti-de Sitter space (hyperbolic space).
	\end{theorem}

	\begin{remark}
		The spin condition in Theorem~\ref{coroasymp} can be ruled out in the three-dimensional case; see more about the positive mass theorem in \cite{chrusciel,wang}.
	\end{remark}

	The theorem below shows that relaxing the boundary conditions of Theorem~\ref{coroasymp} leads us to the same rigidity result. The result we present below resonates with the work \cite{brendle}.
	\begin{theorem}\label{ultmm}
		Let $(M^n,g,f)$ be an asymptotically Schwarzschild--anti-de Sitter sub-static manifold with a boundary satisfying
		\begin{equation*}
			\left\{
			\begin{array}{rcll}
				\Delta f&=&nf,\quad\text{in}\quad M,\\\\
				f&=&\kappa,\quad\text{on}\quad \partial M,\\\\
				\dfrac{\partial f}{\partial\nu}&=&-c,\quad\text{on}\quad \partial M.
			\end{array}
			\right.
		\end{equation*}
		where $\kappa>0$ and $c>0$. Then,
		\begin{equation*}
			-\kappa(n-2)\omega_{n-1}\mathfrak M\geq c\kappa\operatorname{Area}(\partial M)+c^2\int_{\partial M}\dfrac{H}{n-1}\,da.
		\end{equation*}
		Moreover, if $(M^n,g)$ is a spin manifold with a mean-convex boundary, it must be isometric to hyperbolic space.
	\end{theorem}

	\section{Examples}

	Let us recall some examples of Definition~\ref{defserrin}; see \cite{ambrozio,corvino,miaotam}. The first example is the standard hemisphere, an important compact static manifold with boundary and a fundamental model in the proofs of Theorem~\ref{theoBetazero}, Theorem~\ref{prop11}, and Theorem~\ref{sabo}.
	\begin{example}\label{ex1}
		Let $M^n=\mathbb{S}^{n}_{+}$ with metric $g=dr^2+\sin^2(r)g_{\mathbb{S}^{n-1}}$ be the standard hemisphere with potential given by $f(r)=\cos(r)$, where $r\leq\pi/2$ is the height function. It is well known that $\mathbb{S}^{n}_{+}$ with the standard metric is a static Einstein manifold, i.e., $\mathring{\operatorname{Ric}}=0$. Moreover, $\mathring{\nabla}^2f=0$, $R=n(n-1)$, and $\Delta f=-\dfrac{R}{n-1}f$. Here, $\lambda=n$ and $\beta=0$.
	\end{example}

	The next example also represents a compact static manifold with a disconnected boundary. The first is an example of a static manifold with a disconnected boundary. The results of Theorem~\ref{sabo}, Theorem~\ref{thm:k-sub-static}, and Theorem~\ref{minkhesssemtraco} also apply to compact manifolds with a disconnected boundary.

	\begin{example}\label{example2}
		Recall the following example of a compact manifold with boundary and constant scalar curvature. Consider the cylinder $[0,\pi]\times\mathbb{S}^{n-1}$ with product metric $g=ds^2+(n-2)g_{\mathbb{S}^{n-1}}$, where $g_{\mathbb{S}^{n-1}}$ is the standard metric, $R=n(n-1)$, $\lambda=n$, $\beta=0$, and the potential function is $f(s)=\sin(s)$. Moreover,
		\begin{equation*}
			-f\operatorname{Ric}+\nabla^2f-(\Delta f)g=0.
		\end{equation*}
		The standard cylinder has a parallel Ricci tensor but is not Einstein.
	\end{example}

	Now, we will focus on the classic examples of $V$-static manifolds (cf. \cite{miaotam}). These examples are important for Definition~\ref{vsubstatic}, Theorem~\ref{thm:k-sub-static}, and Theorem~\ref{minkhesssemtraco}.

	\begin{example}\label{miaomiao}
		If a $V$-static manifold on a connected, complete manifold without boundary is Einstein, then $(M^n,g)$ must be isometric to one of the following:
		\begin{enumerate}
			\item the standard sphere $\mathbb{S}^n$: Consider $g$ as the canonical metric of $\mathbb{S}^n$ and define the function $f$ along a geodesic $\gamma(s)$ emanating from a point $p\in\mathbb{S}^n$ by
			\begin{equation*}
				f(\gamma(s))=\dfrac{1}{n-1}(A\cos r-k),
			\end{equation*}
			where $r$ denotes the geodesic distance from a point $p$ satisfying $\nabla f(p)=0$, and $A$ is a constant.

			\item the Euclidean space $\mathbb{R}^n$: Consider $g$ as the canonical metric of $\mathbb{R}^n$ and define
			\begin{equation*}
				f(x)=\dfrac{1}{n-1}\left(A-\dfrac{k}{2}|x|^2\right),
			\end{equation*}
			where $A$ is a constant and $k\neq0$.

			\item the hyperbolic space $(\mathbb{H}^n,g)$: Consider $g$ as the canonical metric of $\mathbb{H}^n$ and define the function $f$ along a geodesic $\gamma(s)$ emanating from a point $p\in\mathbb{H}^n$ by
			\begin{equation*}
				f(\gamma(s))=\dfrac{1}{n-1}(k-A\cosh r),
			\end{equation*}
			where $r$ denotes the geodesic distance from a point $p$ satisfying $\nabla f(p)=0$, and $A$ is a constant.

			\item a warped product space $(\mathbb{R}\times\Sigma^{n-1},dt^2+\cosh^2(t)g_0)$, where $(\Sigma^{n-1},g_0)$ is an Einstein manifold, i.e.,
			$\operatorname{Ric}_{g_0}=-(n-2)g_0$ and
			\begin{equation*}
				f(t,x)=A\sinh(t)+\dfrac{1}{n-1},
			\end{equation*}
			where $A>0$ and $x\in\Sigma$.
		\end{enumerate}
	\end{example}

	The final example we will present is fundamental to Theorem~\ref{thm:k-sub-static} (and Theorem~\ref{minkhesssemtraco}). We will show all computations for this example to make the proof of the theorems easy to follow; see \cite{miaotam}.

	\begin{example}\label{esferanaesfera}[Sharpness of Theorem~\ref{thm:k-sub-static} and Theorem~\ref{minkhesssemtraco}]
		Let \(Q^n(\varepsilon)\), \(n\geq3\), be the simply connected space form of constant sectional curvature \(\varepsilon\in\mathbb{R}\), and take the functions
	\begin{equation*}
		\operatorname{S}_{\varepsilon}(r)=
		\begin{cases}
			\dfrac{1}{\sqrt{\varepsilon}}\sin(\sqrt{\varepsilon}r),&\varepsilon>0,\\
			r,&\varepsilon=0,\\
			\dfrac{1}{\sqrt{-\varepsilon}}\sinh(\sqrt{-\varepsilon}r),&\varepsilon<0,
		\end{cases}
		\qquad
		\operatorname{C}_{\varepsilon}(r)=\operatorname{S}_{\varepsilon}'(r).
	\end{equation*}
	Consider
	\begin{equation*}
		M=\{r\leq r_0\}\subset Q^n(\varepsilon),
		\qquad
		g=dr^2+\operatorname{S}_{\varepsilon}(r)^2g_{\mathbb{S}^{n-1}}.
	\end{equation*}
	If \(\varepsilon>0\), assume that
	\begin{equation*}
		0<r_0<\dfrac{\pi}{2\sqrt{\varepsilon}}
	\end{equation*}
    to obtain $f>0$ in $M$.
	Take the constants \(\alpha\) and \(\beta\). We set the following cases to be considered:
	\begin{enumerate}
		\item if \(\varepsilon>0\), assume that
		\begin{equation*}
			\alpha>0,\qquad \beta\geq0,\qquad \beta<n\varepsilon\alpha\operatorname{C}_{\varepsilon}(r_0);
		\end{equation*}
		\item if \(\varepsilon=0\), assume that
		\begin{equation*}
			\beta>0,\qquad \alpha>\dfrac{\beta r_0^2}{2n};
		\end{equation*}
		\item if \(\varepsilon<0\), assume that
		\begin{equation*}
			\alpha<0,\qquad \beta>n\varepsilon\alpha\operatorname{C}_{\varepsilon}(r_0).
		\end{equation*}
	\end{enumerate}
	Define
	\begin{equation*}
		f(r)=
		\begin{cases}
			\alpha\operatorname{C}_{\varepsilon}(r)-\dfrac{\beta}{n\varepsilon},&\varepsilon\neq0,\\
			\alpha-\dfrac{\beta}{2n}r^2,&\varepsilon=0.
		\end{cases}
	\end{equation*}
	Again, the parameter restrictions guarantee \(f>0\) in \(M\). Moreover, the nature of the constant \(\alpha\) was presented by \cite{miaotam}. Since
	\begin{equation*}
		\operatorname{C}_{\varepsilon}'(r)=-\varepsilon\operatorname{S}_{\varepsilon}(r),
		\qquad
		\operatorname{C}_{\varepsilon}(r)^2+\varepsilon\operatorname{S}_{\varepsilon}(r)^2=1,
	\end{equation*}
	the radial Hessian formula gives
	\begin{equation*}
		\nabla^2f=-\varepsilon fg-\dfrac{\beta}{n}g,
		\qquad
		\Delta f=-n\varepsilon f-\beta.
	\end{equation*}
	Thus, for \(\lambda=n\varepsilon\),
	\begin{equation*}
		\mathring{\nabla}^2f=\nabla^2f-\dfrac{\Delta f}{n}g=0.
	\end{equation*}
	Moreover, using \(\operatorname{Ric}=(n-1)\varepsilon g\), we obtain
	\begin{align*}
		f\operatorname{Ric}-\nabla^2f+(\Delta f)g=-\dfrac{(n-1)\beta}{n}g.
	\end{align*}
	Consequently,
	\begin{equation*}
		-\mathfrak{L}_g^{\ast}(f)+\dfrac{(n-1)\beta}{n}g=0,
	\end{equation*}
	and hence \((M^n,g,f,\beta)\) is a \(V\)-sub-static manifold.

	Along \(\partial M=\{r=r_0\}\), the outward unit normal vector is \(\nu=\partial_r\). If
	\begin{equation*}
		\kappa:=f|_{\partial M}=f(r_0)
	\end{equation*}
	and
	\begin{equation*}
		q:=-\dfrac{\partial f}{\partial\nu}=
		\begin{cases}
			\alpha\varepsilon\operatorname{S}_{\varepsilon}(r_0),&\varepsilon\neq0,\\
			\dfrac{\beta r_0}{n},&\varepsilon=0,
		\end{cases}
	\end{equation*}
	then \(\kappa>0\) and \(q>0\). Therefore, the constant Neumann and Robin boundary conditions are satisfied with
	\begin{equation*}
		\dfrac{\partial f}{\partial\nu}=-q
		\qquad\mbox{and}\qquad
		\dfrac{\partial f}{\partial\nu}=-cf,
		\qquad
		c=\dfrac{q}{\kappa}>0.
	\end{equation*}
	Since \(f\) is constant on \(\partial M\), we have \(\nabla_{\partial M}f=0\). The second fundamental form and the mean curvature of \(\partial M\) are
	\begin{equation*}
		A=\dfrac{\operatorname{C}_{\varepsilon}(r_0)}{\operatorname{S}_{\varepsilon}(r_0)}g_{\partial M},
		\qquad
		H=(n-1)\dfrac{\operatorname{C}_{\varepsilon}(r_0)}{\operatorname{S}_{\varepsilon}(r_0)}.
	\end{equation*}
	In particular, \(\partial M\) is totally umbilical.

	The divergence theorem gives
	\begin{equation*}
		\lambda\int_Mf\,dv+\beta\operatorname{Vol}(M)
		=-\int_{\partial M}\dfrac{\partial f}{\partial\nu}\,da
		=q\operatorname{Area}(\partial M).
	\end{equation*}
	On the other hand,
	\begin{equation*}
		\lambda\kappa+\beta
		=nq\dfrac{\operatorname{C}_{\varepsilon}(r_0)}{\operatorname{S}_{\varepsilon}(r_0)}.
	\end{equation*}
	Therefore,
	\begin{equation*}
		(\lambda\kappa+\beta)\operatorname{Area}(\partial M)^2
		=n\left(\lambda\int_Mf\,dv+\beta\operatorname{Vol}(M)\right)
		\int_{\partial M}\dfrac{H}{n-1}\,da.
	\end{equation*}
	Furthermore,
	\begin{equation*}
		\int_{\partial M}f\,da=\kappa\operatorname{Area}(\partial M),
		\qquad
		\int_{\partial M}\dfrac{fH}{n-1}\,da
		=\kappa\dfrac{\operatorname{C}_{\varepsilon}(r_0)}{\operatorname{S}_{\varepsilon}(r_0)}
		\operatorname{Area}(\partial M),
	\end{equation*}
	and hence
	\begin{multline*}
		\left(\int_{\partial M}f\,da\right)
		\left(\lambda\int_{\partial M}f\,da+\beta\operatorname{Area}(\partial M)\right)\\
		=n\left(\lambda\int_Mf\,dv+\beta\operatorname{Vol}(M)\right)
		\int_{\partial M}\dfrac{fH}{n-1}\,da.
	\end{multline*}
	Consequently, equality holds in the Minkowski-type inequality for \(V\)-sub-static manifolds.
	\end{example}

	\section{Preliminaries}\label{back}

	In this section, we present the results needed to prove our theorems; we also provide proofs for some of them for the reader's convenience. Moreover, we provide important lemmas that we use repeatedly.

	\subsection{Reilly's Generalized Formulas}

	Recall the generalized Reilly integral formula proved by \cite{li,qiu}, which has been used in the study of many Einstein-type manifolds, such as $V$-static, static, and electrostatic manifolds. This formula will be applied in the proof of Theorem~\ref{prop11}.

	\begin{theorem}\label{li}\cite{li}
		Let $(M^n, g)$ be an $n$-dimensional smooth Riemannian manifold with a smooth boundary $\partial M$. Given two smooth functions $f$ and $U$ on $M^n$ for which $\dfrac{\nabla^2f}{f}$ is continuous up to the boundary, we have

		\begin{align*}
			&\int_M\mathfrak{L}_g^{\ast}(f)\left(\nabla U-\dfrac{\nabla f}{f}U, \nabla U - \dfrac{\nabla f}{f}U\right)\,dv\\
			&+\int_M \left[f\left(\Delta U-\dfrac{\Delta f}{f}U\right)^2-f\left|\nabla^2U-\dfrac{\nabla^2f}{f}U\right|^2\right]\,dv\\
			&=\int_{\partial M} \left(fA(\nabla_{\partial M} U,\nabla_{\partial M} U)+ 2 f\dfrac{\partial U}{\partial\nu}\Delta_{\partial M} U + fH\left(\dfrac{\partial U}{\partial\nu}\right)^2+\dfrac{\partial f}{\partial\nu}|\nabla_{\partial M} U|^2\right)\,da\\
			&+\int_{\partial M}\left(-2U\dfrac{\partial U}{\partial\nu}\left(\Delta_{\partial M} f + H\dfrac{\partial f}{\partial\nu}\right)-U^2\dfrac{\nabla^2f-\Delta f g}{f}(\nabla f, \nu) + 2U\nabla^2f(\nabla_{\partial M} U,\nu) \right)\,da.
		\end{align*}
		Here, $\nu$ is the outward unit normal vector, $A(\cdot, \cdot)$ and $H$ are the second fundamental form and the mean curvature of $\partial M$, respectively. Moreover,
		\begin{align*}
			-\mathfrak{L}_g^{\ast}(f):=(\Delta f) g-\nabla^2f + f \operatorname{Ric}.
		\end{align*}
	\end{theorem}

	The above formula is closely related to Theorem~\ref{divkato}. Clearly, when $f= 1$ and $\mathfrak{L}_g^{\ast}(f) = - \operatorname{Ric}$, the above formula recovers Reilly's original formula. Just as nonnegative Ricci curvature in the interior integral is crucial for applications of Reilly's original formula, the nonnegative-definite tensor $-\mathfrak{L}_g^{\ast}(f)$ plays a crucial role in the generalized formula. Consequently, we also have the following proposition (cf. \cite[Remark 3.2]{li}).

	\begin{proposition}[Generalized Reilly Identity \cite{qiu}]\label{prop1}
		Let $(M^n,g)$ be a compact Riemannian manifold with boundary $\partial M$. Given two functions $f$ and $U$ on $M^n$ and a constant $k$, we have
		\begin{multline}
			\int_Mf[(\Delta U +knU)^2 - |\nabla^2U+kUg|^{2}]\,dv = (n-1)k\int_M(\Delta f +nkf)U^2\,dv\nonumber\\
			+\int_{\partial M}f\left[2\dfrac{\partial U}{\partial\nu}\Delta_{\partial M}U + (n-1)H\left(\dfrac{\partial U}{\partial\nu}\right)^2 + A(\nabla_{\partial M}U,\nabla_{\partial M}U) + 2(n-1)kU\dfrac{\partial U}{\partial\nu}\right]\,da\\
			+\int_M(\nabla^2f-(\Delta f)g - 2(n-1)kfg +f\operatorname{Ric})(\nabla U,\nabla U)\,dv \\
			+\int_{\partial M}\dfrac{\partial f}{\partial\nu}[|\nabla_{\partial M}U|^2 -(n-1)kU^2]\,da.\\
		\end{multline}
	\end{proposition}

	\subsection{Properties of static manifolds}
	We start with a proposition from \cite{freitas} that establishes basic properties of electrostatic manifolds. The following proposition is important for the proof of Theorem~\ref{sabo} and Theorem~\ref{theoasymptotic}.

	\begin{proposition}\cite{freitas}\label{properties}
		Let $(M^n,g,f,E)$ be an electrostatic system with a nonempty boundary. Then the following assertions hold:

		\begin{itemize}
			\item[(i)] The scalar curvature of $(M^n,g)$ is given by
			\begin{equation}\label{rrr}
				R=2(\lambda+|E|^2).
			\end{equation}

			\item[(ii)] The boundary $\partial M$ is totally geodesic. In particular,
			\begin{equation}\label{gausseq}
				\mathring{\operatorname{Ric}}(\nu,\nu)=\dfrac{n-2}{2n}R-\dfrac{1}{2}R^{\partial M},
			\end{equation}
			where $\mathring{\operatorname{Ric}}=\operatorname{Ric}-\dfrac{R}{n}g$ is the traceless Ricci tensor, and $R^{\partial M}$ is the scalar curvature of $\partial M$.

			\item[(iii)] If $\partial M=\displaystyle\bigcup_{i=1}^{l}\Sigma_i$, where $\Sigma_i$ are the connected components of $\partial M$, then $\kappa_i:=|\nabla f|\big|_{\Sigma_i}$ are non-null constant (they are called surface gravities);

			\item[(iv)] $\nabla f$ and $E$ are proportional along $\partial M$. In particular, $|\langle E,\nabla f\rangle|=|E|\,|\nabla f|$ along $\partial M$;

			\item[(v)] If $M$ is compact and $|E|$ is constant, then
			\begin{equation*}
				\lambda>(n-2)|E|^2.
			\end{equation*}
			In particular,
			\begin{equation*}
				2\lambda\leq R<\dfrac{2(n-1)}{n-2}\lambda.
			\end{equation*}

			\item[(vi)] $(M^n,g,f)$ is sub-static, i.e.,
			\begin{equation*}
				f\operatorname{Ric}-\nabla^2f+(\Delta f)g=2f(|E|^2g-E^{\flat}\otimes E^{\flat}).
			\end{equation*}
		\end{itemize}
	\end{proposition}

	The following positive mass theorem will be essential for the proof of the rigidity results for noncompact manifolds presented in the introduction. Boucher, Gibbons, and Horowitz \cite{boucher1} proved that a three-dimensional asymptotically hyperbolic static manifold with a negative cosmological constant has nonpositive ADM mass, i.e., $\mathfrak M\leq0$. Combined with the positive mass theorem, this result forces the mass to vanish and yields the rigidity of the static manifold with a negative cosmological constant; namely, it must be hyperbolic space. We will extend this argument to sub-static manifolds.

	\begin{theorem}\label{PMT}\cite{chrusciel,wang}
		Let $(M^n,g)$ be an asymptotically hyperbolic spin manifold such that $R\geq-n(n-1)$. Therefore, $\mathfrak M\geq0$. Moreover, equality holds if and only if $(M^n,g)$ is isometric to hyperbolic space.
	\end{theorem}

	To prove Theorem~\ref{coroasymp}, we need the following lemmas. We first present a standard computation in the context of the Schwarzschild--anti-de Sitter manifold, related to the positive mass theorem stated above; this computation can also be found in \cite{boucher1}. We will use it repeatedly throughout the paper.

	\begin{lemma}\label{lem:schwarz-vector}
		Let $(M^n,g,f)$ be a Riemannian manifold with a smooth function $f$ such that $\Delta f=-\lambda f-\beta$, where $\lambda,\beta\in\mathbb{R}$. Assume that $(M^n,g,f)$ is asymptotically Schwarzschild--anti-de Sitter. Therefore, if $\eta=f\partial_r$ denotes the outward unit normal to the coordinate sphere $S_\rho=\{r=\rho\}$, then
		\begin{multline*}
			\lim_{\rho\to\infty}\int_{S_\rho}\left\langle\dfrac{1}{2f}\nabla|\nabla f|^2-\dfrac{\Delta f}{nf}\nabla f,\eta\right\rangle\,da \\
			=\dfrac{\omega_{n-1}(n-1)(n-2)\lambda\mathfrak M}{n} - \dfrac{\omega_{n-1}\lambda}{n^2}\lim_{\rho\to\infty}\left[\dfrac{\beta\rho^n}{f(\rho)}\right].
		\end{multline*}
		Therefore, the convergence of the limit depends on the choice of $\beta$.
	\end{lemma}

	\begin{remark}
		The idea behind adding the constant $\beta$ is to explore whether it is possible to extend the proof of Theorem~\ref{coroasymp} to $V$-sub-static manifolds. A straightforward computation shows that asymptotically hyperbolic $V$-sub-static manifolds are hyperbolic if $\beta\leq0$.
	\end{remark}

	\begin{proof}[Proof of Lemma~\ref{lem:schwarz-vector}]
		Consider the Schwarzschild--anti-de Sitter metric
		\begin{equation*}
			g=f^{-2}dr^2+r^2g_{\mathbb{S}^{n-1}},
		\end{equation*}
		whose static potential is
		\begin{equation*}
			f(r)=\sqrt{1-\dfrac{2\mathfrak M}{r^{n-2}}-\dfrac{\lambda r^2}{n}}.
		\end{equation*}
		Taking the derivative, we obtain
		\begin{equation*}
			f'=\dfrac{1}{f}\left(\dfrac{n-2}{r^{n-1}}\mathfrak M-\dfrac{\lambda r}{n}\right).
		\end{equation*}
		Since $g^{rr}=f^2$,
		\begin{equation*}
			\nabla f=g^{rr}f'\partial_r=f\left(\dfrac{n-2}{r^{n-1}}\mathfrak M-\dfrac{\lambda r}{n}\right)\partial_r.
		\end{equation*}

		Setting
		\begin{equation*}
			A(r)=\dfrac{n-2}{r^{n-1}}\mathfrak M-\dfrac{\lambda r}{n},
		\end{equation*}
		we have
		\begin{equation*}
			\nabla f=fA\,\partial_r,\qquad
			|\nabla f|^2=A^2,\qquad
			A'=-\dfrac{(n-1)(n-2)\mathfrak M}{r^n}-\dfrac{\lambda}{n}.
		\end{equation*}
		Hence
		\begin{equation*}
			\dfrac{1}{2f}\nabla|\nabla f|^2=fAA'\partial_r,\qquad
			-\dfrac{\Delta f}{nf}\nabla f=\dfrac{\lambda f+\beta}{n}A\partial_r.
		\end{equation*}
		Therefore,
		\begin{align*}
			\dfrac{1}{2f}\nabla|\nabla f|^2-\dfrac{\Delta f}{nf}\nabla f
			&=A\left(fA'+\dfrac{\lambda f+\beta}{n}\right)\partial_r\\
			&=A\left(-\dfrac{(n-1)(n-2)\mathfrak Mf}{r^n}+\dfrac{\beta}{n}\right)\partial_r,
		\end{align*}
		proving the first identity.

		Now let $\eta=f\partial_r$ denote the outward unit normal to the coordinate sphere $S_\rho=\{r=\rho\}$. Since $g(\partial_r,\partial_r)=f^{-2}$, we have $\langle\partial_r,\eta\rangle=\dfrac{1}{f}$. Consequently,
		\begin{equation*}
			\left\langle\dfrac{1}{2f}\nabla|\nabla f|^2-\dfrac{\Delta f}{nf}\nabla f,\eta\right\rangle
			=\left(\dfrac{(n-2)\mathfrak M}{r^{n-1}}-\dfrac{\lambda r}{n}\right)
			\left(-\dfrac{(n-1)(n-2)\mathfrak M}{r^n}+\dfrac{\beta}{nf}\right).
		\end{equation*}
		Using $\operatorname{Area}(S_\rho)=\omega_{n-1}\rho^{n-1}$, where $\omega_{n-1}$ is the area of the standard sphere $\mathbb{S}^{n-1}$, we conclude that
		\begin{multline*}
			\int_{S_\rho}\left\langle\dfrac{1}{2f}\nabla|\nabla f|^2-\dfrac{\Delta f}{nf}\nabla f,\eta\right\rangle\,da\\
			=\omega_{n-1}\rho^{n-1}
			\left(\dfrac{(n-2)\mathfrak M}{\rho^{n-1}}-\dfrac{\lambda\rho}{n}\right)
			\left(-\dfrac{(n-1)(n-2)\mathfrak M}{\rho^n}+\dfrac{\beta}{nf(\rho)}\right).
		\end{multline*}
		Letting $\rho\to+\infty$, we obtain the result. In fact,
		\begin{multline*}
			\lim_{\rho\to+\infty}\int_{S_\rho}\left\langle\dfrac{1}{2f}\nabla|\nabla f|^2-\dfrac{\Delta f}{nf}\nabla f,\eta\right\rangle\,da \\
			=\omega_{n-1}\lim_{\rho\to\infty}\left[\dfrac{(n-2)\beta\mathfrak M}{nf(\rho)}+\dfrac{(n-1)(n-2)\lambda\mathfrak M}{n}\right. \\
			\left.-\dfrac{(n-1)(n-2)^2\mathfrak M^2}{\rho^n}-\dfrac{\beta\lambda\rho^n}{n^2f(\rho)}\right].
		\end{multline*}
	\end{proof}

	\begin{lemma}\label{lem:reissner-vector}
		Let $(M^n,g,f,E)$ be an asymptotically Reissner--Nordstr\"om--anti-de Sitter electrostatic manifold. For $S_\rho=\{r=\rho\}$, let $\eta=f\partial_r$ be its outward unit normal and set
		\begin{align*}
			X &= \operatorname{Ric}(\nabla f)+\dfrac{2}{n}\bigl((n-2)|E|^2-\lambda\bigr)\nabla f, \\
			Y &= \bigl(|E|^2g-E^\flat\otimes E^\flat\bigr)(\nabla f).
		\end{align*}
		Then,
		\begin{equation*}
			\lim_{\rho\to+\infty}\int_{S_\rho}\langle X-2Y,\eta\rangle\,da=2\omega_{n-1}(n-2)\mathfrak M\dfrac{\lambda}{n}.
		\end{equation*}
		Moreover,
		\begin{equation*}
			\lim_{\rho\to+\infty}\int_{S_\rho}|E|^2\langle\nabla f,\eta\rangle\,da=0
		\end{equation*}
		and
		\begin{equation*}
			\lim_{\rho\to+\infty}\int_{S_\rho}\bigl(f\partial_\eta|E|^2-|E|^2\partial_\eta f\bigr)\,da=0.
		\end{equation*}
	\end{lemma}

	\begin{proof}[Proof of Lemma~\ref{lem:reissner-vector}]
		The electrostatic structural equation gives
		\begin{equation*}
			X-2Y=\dfrac{1}{2f}\nabla|\nabla f|^2-\dfrac{\Delta f}{nf}\nabla f.
		\end{equation*}
		By \eqref{RNDS} and \eqref{E-vector},
		\begin{align*}
			\nabla f = f\biggl(\dfrac{(n-2)\mathfrak M}{r^{n-1}}-\dfrac{(n-2)Q^2}{r^{2n-3}}-\dfrac{2\lambda r}{n(n-1)}\biggr)\partial_r, \qquad
			|E|^2    = \dfrac{(n-1)(n-2)Q^2}{2r^{2(n-1)}}.
		\end{align*}
		Therefore,
		\begin{align*}
			\dfrac{\Delta f}{nf}\nabla f
			= \dfrac{2}{n(n-1)}\bigl((n-2)|E|^2-\lambda\bigr)\nabla f = \biggl(\dfrac{(n-2)^2Q^2}{nr^{2(n-1)}}-\dfrac{2\lambda}{n(n-1)}\biggr)\nabla f.
		\end{align*}
		Consequently,
		\begin{multline*}
			\biggl\langle\dfrac{1}{2f}\nabla|\nabla f|^2-\dfrac{\Delta f}{nf}\nabla f,\eta\biggr\rangle \\
			= (n-1)(n-2)\biggl(\dfrac{2(n-1)Q^2}{nr^{2(n-1)}}-\dfrac{\mathfrak M}{r^n}\biggr)
			\biggl(\dfrac{(n-2)\mathfrak M}{r^{n-1}}-\dfrac{(n-2)Q^2}{r^{2n-3}}-\dfrac{2\lambda r}{n(n-1)}\biggr) \\
			= \dfrac{2(n-2)(n-1)^2Q^2}{n}\biggl(\dfrac{(n-2)\mathfrak M}{r^{3(n-1)}}-\dfrac{(n-2)Q^2}{r^{4n-5}}-\dfrac{2\lambda}{n(n-1)r^{2n-3}}\biggr) \\
			- (n-1)(n-2)\mathfrak M\biggl(\dfrac{(n-2)\mathfrak M}{r^{2n-1}}-\dfrac{(n-2)Q^2}{r^{3(n-1)}}-\dfrac{2\lambda}{n(n-1)r^{n-1}}\biggr).
		\end{multline*}
		Since $\operatorname{Area}(S_\rho)=\omega_{n-1}\rho^{n-1}$, we obtain
		\begin{multline*}
			\int_{S_\rho}\langle X-2Y,\eta\rangle\,da \\
			= \dfrac{2(n-2)(n-1)^2Q^2\omega_{n-1}}{n}\rho^{n-1}
			\biggl(\dfrac{(n-2)\mathfrak M}{\rho^{3(n-1)}}-\dfrac{(n-2)Q^2}{\rho^{4n-5}}-\dfrac{2\lambda}{n(n-1)\rho^{2n-3}}\biggr) \\
			- (n-1)(n-2)\mathfrak M\omega_{n-1}\rho^{n-1}
			\biggl(\dfrac{(n-2)\mathfrak M}{\rho^{2n-1}}-\dfrac{(n-2)Q^2}{\rho^{3(n-1)}}-\dfrac{2\lambda}{n(n-1)\rho^{n-1}}\biggr).
		\end{multline*}

		Moreover,
		\begin{multline*}
			\int_{S_\rho}|E|^2\langle\nabla f,\eta\rangle\,da \\
			= \dfrac{\omega_{n-1}(n-1)(n-2)Q^2}{2}
			\biggl(\dfrac{(n-2)\mathfrak M}{\rho^{n-1}}-\dfrac{(n-2)Q^2}{\rho^{2n-3}}-\dfrac{2\lambda\rho}{n(n-1)}\biggr)\dfrac{1}{\rho^{n-1}},
		\end{multline*}
		which converges to zero. Finally,
		\begin{equation*}
			\nabla|E|^2=-(n-1)^2(n-2)\dfrac{Q^2}{r^{2n-1}}f^2\partial_r.
		\end{equation*}
		Thus,
		\begin{multline*}
			\int_{S_\rho}\bigl(f\partial_\eta|E|^2-|E|^2\partial_\eta f\bigr)\,da
			= -(n-1)^2(n-2)\omega_{n-1}Q^2\dfrac{f(\rho)^2}{\rho^n} \\
			- \dfrac{(n-1)(n-2)\omega_{n-1}Q^2}{2}
			\biggl(\dfrac{(n-2)\mathfrak M}{\rho^{2n-2}}-\dfrac{(n-2)Q^2}{\rho^{3n-4}}-\dfrac{2\lambda}{n(n-1)\rho^{n-2}}\biggr),
		\end{multline*}
		which converges to zero since $f(\rho)\to\rho$ as $\rho\to+\infty$.
	\end{proof}

	In what follows, we will prove a Penrose-like inequality with charge and a cosmological constant for an asymptotically Reissner--Nordstr\"om--anti-de Sitter electrostatic manifold. The formulas and techniques used in the proof of Theorem~\ref{theoasymptotic} will also yield the proofs of Theorem~\ref{sabo} and Theorem~\ref{coroasymp}.

	\begin{theorem}\label{theoasymptotic}
		Let $(M^{n},g,f,E)$ be an asymptotically Reissner--Nordstr\"om--anti-de Sitter electrostatic manifold satisfying
		\begin{equation*}
			\left\{
			\begin{array}{rcll}
				\Delta f &=& \dfrac{2}{n-1}((n-2)|E|^2-\lambda)f,\quad\text{in}\quad M,\\\\
				d(fE^{\flat}) &=& 0,\quad\text{in}\quad M,\\\\
				f &=& 0,\quad\text{on}\quad \partial M,\\\\
				\dfrac{\partial f}{\partial\nu} &=& -c,\quad\text{on}\quad \partial M.
			\end{array}
			\right.
		\end{equation*}
In particular, if $\operatorname{div}E=0$, we have
		\begin{multline*}\label{only E}
			c\int_{\partial M}\dfrac{R^{\partial M}}{2}\,da + 2\omega_{n-1}(n-2)\mathfrak M\dfrac{\lambda}{n} 
			\geq \dfrac{(n-2)c\lambda}{n}\operatorname{Area}(\partial M)\\
            + \dfrac{c\omega_{n-1}^{2}(n-1)(n-2)}{2\operatorname{Area}(\partial M)} Q(\partial M)^2 
            + \dfrac{c(n-2)\omega_{n-1}^2(n-1)(n-2)}{n\operatorname{Area}(\partial M)}Q(\partial M)^2\\
            -\dfrac{4(n-2)}{n}\int_Mf\bigl(|\nabla E|^2+\operatorname{Ric}(E,E)\bigr)\,dv.
		\end{multline*}
		Here, $\omega_{n-1}$ denotes the area of the standard sphere $\mathbb{S}^{n-1}$. Equality holds if and only if $E$ is parallel to $\nu$ and $\mathring{\nabla}^2f=0$.
	\end{theorem}

	\begin{corollary}
		Let $(M^{3},g,f,E)$ be an asymptotically Reissner--Nordstr\"om--anti-de Sitter electrostatic manifold with boundary satisfying
		\begin{equation*}
			\left\{
			\begin{array}{rcll}
				\Delta f &=& (|E|^2-\lambda)f,\quad\text{in}\quad M,\\\\
				d(fE^{\flat}) &=& 0,\quad\text{in}\quad M,\\\\
				f &=& 0,\quad\text{on}\quad \partial M,\\\\
				\dfrac{\partial f}{\partial\nu} &=& -c,\quad\text{on}\quad \partial M.
			\end{array}
			\right.
		\end{equation*}
        Assume that \(\operatorname{div}(E)=0\) and \[\dfrac{8c\pi^2Q(\partial M)^2}{\operatorname{Area}(\partial M)}
           \ge \int_Mf\bigl(|\nabla E|^2+\operatorname{Ric}(E,E)\bigr)\,dv.\]
   Then,
		\begin{align*}
			2\pi\mathcal{X}(\partial M) +  \dfrac{8}{3}\pi\lambda\mathfrak M
			\geq \dfrac{\lambda}{3}\operatorname{Area}(\partial M)  + \dfrac{16\pi^2 Q(\partial M)^2}{\operatorname{Area}(\partial M)}.
		\end{align*}
        In particular, if $\lambda=0$, then we have
    		\begin{align*}
			\operatorname{Area}(\partial M)
			\geq 4\pi Q(\partial M)^2.
		\end{align*}
	\end{corollary}

	The equality in Theorem~\ref{theoasymptotic} cannot occur for the \((n+1)\)-dimensional Reissner--Nordstr\"om--anti-de Sitter manifold because, for this model, we do not have $\mathring{\nabla}^2f =0$, which is necessary for the equality in Theorem~\ref{divkato}. 

	The preceding corollary illustrates how the integral of the Ricci curvature in the direction of $E$ has topological implications for the boundary.

	\begin{proof}[Proof of Theorem~\ref{theoasymptotic}]
		Recall Proposition~\ref{properties}. For $\rho>0$, let $M_\rho=M\cap\{r\leq\rho\}$. Its boundary is $\partial M_\rho=\partial M\cup S_\rho$, with outward unit normals $\nu$ and $\eta$, respectively. Let
		\begin{align*}
			X&=\operatorname{Ric}(\nabla f)+\dfrac{2}{n}\bigl((n-2)|E|^2-\lambda\bigr)\nabla f,\\
			Y&=\bigl(|E|^2g-E^\flat\otimes E^\flat\bigr)(\nabla f).
		\end{align*}
		From Theorem~\ref{divkato} and the electrostatic system given in Definition~\ref{substatic}, we obtain
		\begin{align*}
			\operatorname{div}\left(\dfrac{1}{2f}\nabla|\nabla f|^2-\dfrac{\Delta f}{nf}\nabla f\right)\geq\dfrac{2}{f}(|E|^2|\nabla f|^2-\langle\nabla f,E\rangle^2)+\dfrac{2(n-2)}{n}\langle\nabla|E|^2,\nabla f\rangle.
		\end{align*}
		Moreover, the electrostatic structure gives
		\begin{align*}
			\operatorname{div}\left(\operatorname{Ric}(\nabla f)+\dfrac{(n-1)\Delta f}{nf}\nabla f\right)
			=\operatorname{div}\left(\dfrac{1}{2f}\nabla|\nabla f|^2-\dfrac{\Delta f}{nf}\nabla f\right)
			+2\operatorname{div}((|E|^2g-E^{\flat}\otimes E^{\flat})(\nabla f)).
		\end{align*}
		Therefore,
		\begin{multline}\label{divelectrostatic}
			\operatorname{div}\left(\operatorname{Ric}(\nabla f)+\dfrac{2}{n}((n-2)|E|^2-\lambda)\nabla f\right)\\
			-2\operatorname{div}((|E|^2g-E^{\flat}\otimes E^{\flat})(\nabla f))
			\geq\dfrac{2(n-2)}{n}\langle\nabla|E|^2,\nabla f\rangle.
		\end{multline}
		Hence, the electrostatic divergence inequality is
		\begin{equation*}
			\operatorname{div}X-2\operatorname{div}Y\geq\dfrac{2(n-2)}{n}\langle\nabla|E|^2,\nabla f\rangle.
		\end{equation*}

		Applying the divergence theorem to $M_\rho$ and integrating by parts, we obtain
		\begin{align*}
			\int_{M_\rho}\langle\nabla|E|^2,\nabla f\rangle\,dv
			&= -\int_{M_\rho}|E|^2\Delta f\,dv+\int_{\partial M_\rho}|E|^2\partial_Nf\,da \\
			&= -\int_{M_\rho}|E|^2\Delta f\,dv-c\int_{\partial M}|E|^2\,da+\int_{S_\rho}|E|^2\langle\nabla f,\eta\rangle\,da,
		\end{align*}
		where $N$ is the outward unit normal to $\partial M_\rho$. Hence,
		\begin{multline}\label{eq:rho-divergence}
			\int_{S_\rho}\langle X-2Y,\eta\rangle\,da+\int_{\partial M}\langle X-2Y,\nu\rangle\,da \\
			\geq -\dfrac{2(n-2)}{n}\int_{M_\rho}|E|^2\Delta f\,dv-\dfrac{2c(n-2)}{n}\int_{\partial M}|E|^2\,da \\
			+ \dfrac{2(n-2)}{n}\int_{S_\rho}|E|^2\langle\nabla f,\eta\rangle\,da.
		\end{multline}
		Since $f=0$ on $\partial M$, the electrostatic equation implies that $E$ is normal to $\partial M$; see Proposition~\ref{properties}. Using $\nabla f=-c\nu$, we obtain
		\begin{align*}
			\langle Y,\nu\rangle
			= |E|^2\langle\nabla f,\nu\rangle-\langle E,\nabla f\rangle\langle E,\nu\rangle = -c|E|^2+c|E|^2=0
		\end{align*}
		and
		\begin{align*}
			\langle X,\nu\rangle
			&= \operatorname{Ric}(\nabla f,\nu)+\dfrac{2}{n}\bigl((n-2)|E|^2-\lambda\bigr)\langle\nabla f,\nu\rangle \\
			&= -c\operatorname{Ric}(\nu,\nu)-\dfrac{2c}{n}\bigl((n-2)|E|^2-\lambda\bigr).
		\end{align*}
		Combining the above identities with \eqref{eq:rho-divergence} yields
		\begin{multline*}
			\int_{S_\rho}\langle X-2Y,\eta\rangle\,da+\dfrac{2c\lambda}{n}\operatorname{Area}(\partial M) \\
			\geq c\int_{\partial M}\operatorname{Ric}(\nu,\nu)\,da-\dfrac{2(n-2)}{n}\int_{M_\rho}|E|^2\Delta f\,dv \\
			+ \dfrac{2(n-2)}{n}\int_{S_\rho}|E|^2\langle\nabla f,\eta\rangle\,da.
		\end{multline*}

		Letting $\rho\to+\infty$ and using Lemma~\ref{lem:reissner-vector}, we obtain
		\begin{equation}\label{eq:rho-ricci-boundary}
			\dfrac{2c\lambda}{n}\operatorname{Area}(\partial M)+2\omega_{n-1}(n-2)\mathfrak M\dfrac{\lambda}{n}\geq c\int_{\partial M}\operatorname{Ric}(\nu,\nu)\,da-\dfrac{2(n-2)}{n}\int_M|E|^2\Delta f\,dv.
		\end{equation}
		The Gauss equation together with $R=2(\lambda+|E|^2)$ gives
		\begin{align*}
			\operatorname{Ric}(\nu,\nu)
			= |E|^2+\lambda-\dfrac{R^{\partial M}}{2}.
		\end{align*}
		Combining \eqref{eq:rho-ricci-boundary} with the above identity yields
		\begin{multline*}
			c\int_{\partial M}\dfrac{R^{\partial M}}{2}\,da+2\omega_{n-1}(n-2)\mathfrak M\dfrac{\lambda}{n} \\
			\geq \dfrac{c(n-2)\lambda}{n}\operatorname{Area}(\partial M)+c\int_{\partial M}|E|^2\,da-\dfrac{2(n-2)}{n}\int_M|E|^2\Delta f\,dv.
		\end{multline*}

		Moreover, applying H\"older's inequality to the definition of charge given by \eqref{charge}, we obtain
		\begin{multline}\label{QE}
			\dfrac{1}{2}\omega_{n-1}^2(n-1)(n-2)Q(\partial M)^2=\left(\int_{\partial M}\langle E,\nu\rangle\,da\right)^2\\
			\leq\operatorname{Area}(\partial M)\int_{\partial M}\langle E,\nu\rangle^2\,da
			\leq\operatorname{Area}(\partial M)\int_{\partial M}|E|^2\,da.
		\end{multline}
		Consequently,
		\begin{multline*}
			c\int_{\partial M}\dfrac{R^{\partial M}}{2}\,da + 2\omega_{n-1}(n-2)\mathfrak M\dfrac{\lambda}{n} \\
			\geq \dfrac{(n-2)c\lambda}{n}\operatorname{Area}(\partial M) + \dfrac{c\omega_{n-1}^{2}(n-1)(n-2)}{2\operatorname{Area}(\partial M)} Q(\partial M)^2 - \dfrac{2(n-2)}{n} \int_M |E|^{2}\Delta f\,dv.
		\end{multline*}

		On the other hand, by \cite[Proposition~9.2.2]{petersen}, if $\operatorname{div}E = 0$, then
		\begin{align*}
			\dfrac{1}{2}\Delta|E|^{2} = |\nabla E|^{2} + \underbrace{\nabla_E(\operatorname{div}E)}_{=0} + \operatorname{Ric}(E,E).
		\end{align*}
		Green's identity on $M_\rho$ gives
		\begin{multline*}
			\int_{M_\rho}\bigl(f\Delta|E|^2-|E|^2\Delta f\bigr)\,dv \\
			= \int_{\partial M}\bigl(f\partial_\nu|E|^2-|E|^2\partial_\nu f\bigr)\,da+\int_{S_\rho}\bigl(f\partial_\eta|E|^2-|E|^2\partial_\eta f\bigr)\,da.
		\end{multline*}
		Since $f=0$ and $\partial_\nu f=-c$ on $\partial M$, we have
		\begin{multline*}
			-\int_{M_\rho}|E|^2\Delta f\,dv \\
			= -2\int_{M_\rho}f\bigl(|\nabla E|^2+\operatorname{Ric}(E,E)\bigr)\,dv+c\int_{\partial M}|E|^2\,da \\
			+ \int_{S_\rho}\bigl(f\partial_\eta|E|^2-|E|^2\partial_\eta f\bigr)\,da.
		\end{multline*}
		Letting $\rho\to+\infty$, using Lemma~\ref{lem:reissner-vector}, and applying \eqref{QE}, we obtain
		\begin{equation*}
			-\int_M|E|^2\Delta f\,dv \geq -2\int_Mf\bigl(|\nabla E|^2+\operatorname{Ric}(E,E)\bigr)\,dv+\dfrac{c\omega_{n-1}^2(n-1)(n-2)}{2\operatorname{Area}(\partial M)}Q(\partial M)^2.
		\end{equation*}
		Multiplying the preceding inequality by $\dfrac{2(n-2)}{n}$, we conclude
		\begin{multline*}
			-\dfrac{2(n-2)}{n}\int_M|E|^2\Delta f\,dv \\
			\geq  \dfrac{(n-2)c\omega_{n-1}^2(n-1)(n-2)}{n\operatorname{Area}(\partial M)}Q(\partial M)^2  \\
            -\dfrac{4(n-2)}{n}\int_Mf\bigl(|\nabla E|^2+\operatorname{Ric}(E,E)\bigr)\,dv.
		\end{multline*}
		Equality holds if and only if $E$ is parallel to $\nu$ and $\mathring{\nabla}^2f=0$.

		Finally,
		\begin{multline*}
			c\int_{\partial M}\dfrac{R^{\partial M}}{2}\,da + 2\omega_{n-1}(n-2)\mathfrak M\dfrac{\lambda}{n} 
			\geq \dfrac{(n-2)c\lambda}{n}\operatorname{Area}(\partial M)\\
            + \dfrac{c\omega_{n-1}^{2}(n-1)(n-2)}{2\operatorname{Area}(\partial M)} Q(\partial M)^2 
            + \dfrac{(n-2)c\omega_{n-1}^2(n-1)(n-2)}{n\operatorname{Area}(\partial M)}Q(\partial M)^2\\
            -\dfrac{4(n-2)}{n}\int_Mf\bigl(|\nabla E|^2+\operatorname{Ric}(E,E)\bigr)\,dv.
		\end{multline*}
	\end{proof}

	\subsection{Concircular Fields}

	Let $(M^n,g)$ be an $n$-dimensional connected Riemannian manifold. Following Tashiro \cite{tashiro}, we call a nonconstant scalar field $f$ on $M$ a concircular scalar field, or simply a concircular field, if it satisfies the equation
	\begin{equation}\label{brinkmann}
		\nabla^2f=\phi g,
	\end{equation}
	where $\nabla$ denotes covariant differentiation with respect to $g$, and $\phi$ is a scalar field called the characteristic function of $f$; see also \cite{brinkmann}. Every concircular field induces a local warped product structure. Adapted to our notation, the result can be stated as follows.

	\begin{theorem}\cite[Theorem 4.3.3]{petersen}\label{lem:tashiro}
		If there is a smooth function $f$ whose Hessian is conformal to the metric, as in \eqref{brinkmann}, then the Riemannian structure is
		locally a warped product around any regular point of $f$. Moreover, if $p$ is a critical point of $f$ such that $\phi(p)\neq0$, then
		\begin{equation*}
			g = dr^2 + \psi^{2}(r)g_{\mathbb{S}^{n-1}}
		\end{equation*}
		holds in some neighborhood of $p$, where $\psi(r)=f'(r)$.
	\end{theorem}

	The next results show that the concircular equation given by
	\begin{equation}\label{RI}
		\nabla^2f=-\dfrac{\beta}{n}g
	\end{equation}
	determines the global geometry of the compact Riemannian manifold $(M^n,g,f)$ with boundary $\partial M = f^{-1}(0)$. More precisely, we first prove that $f$ has a unique critical point, which serves as the center of a natural radial structure; see also \cite[Theorem 1.2]{tashiro} and \cite{petersen}. We then show that every boundary point is at the same geodesic distance from this center, implying that the boundary $\partial M$ is a geodesic sphere and, consequently, that $(M,g)$ is a geodesic ball.

	\begin{lemma}\label{lem:quadratic-geodesic}
		Let $\gamma:[0,a]\longrightarrow M$ be a unit-speed geodesic and define $u(s)=f(\gamma(s))$. Then $u''(s)=-\dfrac{\beta}{n}$. Consequently,
		\begin{equation*}
			u(s)=u(0)+u'(0)s-\dfrac{\beta}{2n}s^2.
		\end{equation*}
	\end{lemma}

	\begin{proof}[Proof of Lemma~\ref{lem:quadratic-geodesic}]
		Since $\gamma$ is unit-speed, $u'(s)=\langle\nabla f,\gamma'(s)\rangle$. Hence,
		\begin{equation*}
			u''(s)=\nabla^2f(\gamma',\gamma')+\langle\nabla f,\nabla_{\gamma'}\gamma'\rangle.
		\end{equation*}
		Recall that $\gamma$ is geodesic. Hence, from \eqref{RI}, we obtain
		\begin{equation*}
			u''(s)=\nabla^2f(\gamma',\gamma')=-\dfrac{\beta}{n}.
		\end{equation*}
		Integrating twice, we obtain
		\begin{equation*}
			u(s)=u(0)+u'(0)s-\dfrac{\beta}{2n}s^2.
		\end{equation*}
	\end{proof}

	\begin{lemma}\label{lem:critical}
		The function $f$ has a unique critical point $p\in M$, which is its global maximum.
	\end{lemma}

	\begin{proof}[Proof of Lemma~\ref{lem:critical}]
		Since $M$ is compact and $f=0$ on $\partial M$ while $f>0$ in the interior of $M$, $f$ attains its maximum at some interior point $p$. Hence,
		\begin{equation*}
			\nabla f(p)=0.
		\end{equation*}

		Suppose that $p=\gamma(0)$ and $q=\gamma(l)$ are distinct critical points. Let $\gamma:[0,l]\longrightarrow M$ be the unit-speed minimizing geodesic joining them. By Lemma~\ref{lem:quadratic-geodesic},
		\begin{equation*}
			u(s)=-\dfrac{\beta}{2n}s^2+As+B.
		\end{equation*}
		Since $u'(0)=0$, we obtain $A=0$. Thus,
		\begin{equation*}
			0=u'(l)=-\dfrac{\beta}{n}l,
		\end{equation*}
		which is impossible because $l>0$. Therefore, the critical point is unique.
	\end{proof}

	\begin{corollary}\label{cor:quadratic-potential}
		Let $p$ be the unique critical point of $f$. If $\gamma:[0,l)\longrightarrow M$ is a unit-speed minimizing geodesic starting at $p$, then
		\begin{equation*}
			f(\gamma(r))=f(p)-\dfrac{\beta}{2n}r^2.
		\end{equation*}
		In particular, if $f$ and $g$ extend smoothly to the boundary $\partial M$,
		\begin{equation*}
			{\dfrac{\beta}{2n}}r_0^2=f(p)
		\end{equation*}
		for all $q\in\partial M$, where $d(p,q)=r_0$.
	\end{corollary}

	\begin{proof}[Proof of Corollary~\ref{cor:quadratic-potential}]
		By Lemma~\ref{lem:quadratic-geodesic},
		\begin{equation*}
			f(\gamma(r))=f(p)+\langle\nabla f(p),\gamma'(0)\rangle r-\dfrac{\beta}{2n}r^2.
		\end{equation*}
		Since $\nabla f(p)=0$, we obtain
		\begin{equation*}
			f(\gamma(r))=f(p)-\dfrac{\beta}{2n}r^2.
		\end{equation*}
		If $q=\gamma(r_0)\in\partial M$, then $f(q)=0$. Since $r_0=d(p,q)$ along the minimizing geodesic,
		\begin{equation*}
			0=f(p)-\dfrac{\beta}{2n}r_0^2,
		\end{equation*}
		and therefore
		\begin{equation*}
			d(p,q)=r_0=\sqrt{\dfrac{2nf(p)}{\beta}}.
		\end{equation*}
	\end{proof}
   
    \begin{remark}\label{rem:tashiro-orientation}
    	The apparent sign discrepancy between Theorem~\ref{lem:tashiro} and Corollary~\ref{cor:quadratic-potential} is due to the opposite orientations adopted for the radial parameter. In Corollary~\ref{cor:quadratic-potential}, the coordinate \(r=d(p,x)\) increases from the critical point \(p\) toward the boundary, and hence
    	\begin{equation*}
    		f'(r)=-\dfrac{\beta}{n}r.
    	\end{equation*}
    	By contrast, the parameter in Theorem~\ref{lem:tashiro} is oriented along the integral curves of \(\nabla f\), which point towards \(p\). Thus, if \(s\) denotes this parameter, then \(\partial_s=-\partial_r\), and consequently
    	\begin{equation*}
    		\dfrac{df}{ds}=-f'(r)=\dfrac{\beta}{n}r.
    	\end{equation*}
    	Therefore, the two formulas are consistent. 
    \end{remark}    
    
	\begin{corollary}\label{cor:boundary-sphere}
		The boundary is the geodesic sphere of radius $r_0$ centered at $p$, namely,
		\begin{equation*}
			\partial M={S}_{r_0}(p).
		\end{equation*}
	\end{corollary}

	\begin{proof}[Proof of Corollary~\ref{cor:boundary-sphere}]
		By Corollary~\ref{cor:quadratic-potential},
		\begin{equation*}
			\partial M\subset {S}_{r_0}(p).
		\end{equation*}
		Conversely, let $x\in {S}_{r_0}(p)$. Applying Corollary~\ref{cor:quadratic-potential} along a minimizing geodesic from $p$ to $x$, we find
		\begin{equation*}
			f(x)=f(p)-\dfrac{\beta}{2n}r_0^2=0.
		\end{equation*}
		Hence, $x\in\partial M =f^{-1}(0)$, and therefore
		$\partial M=S_{r_0}(p)$.
	\end{proof}

	\begin{corollary}\label{cor:geodesicball}
		The manifold is the geodesic ball centered at $p$, namely,
		\begin{equation*}
			M=B_{r_0}(p).
		\end{equation*}
	\end{corollary}

	\begin{proof}[Proof of Corollary~\ref{cor:geodesicball}]
		Let $x\in M$. By Corollary~\ref{cor:quadratic-potential},
		\begin{equation*}
			f(x)=f(p)-\dfrac{\beta}{2n}d(p,x)^2.
		\end{equation*}
		Since $f(x)\geq0$, we obtain $d(p,x)\leq r_0$, and therefore $M\subset B_{r_0}(p)$. Conversely, Corollary~\ref{cor:boundary-sphere} shows that
		$\partial M=S_{r_0}(p)$.
		Hence, every point satisfying $d(p,x)<r_0$ lies in the interior of $M$, so
		$B_{r_0}(p)\subset M$.
		This proves the result.
	\end{proof}

	We conclude this section with an Obata-type characterization of spherical caps due to Almaraz and Barbosa \cite{almaraz}. While the preceding discussion concerns a concircular field satisfying a Dirichlet boundary condition, the following result treats the Robin boundary condition that arises in the rigidity arguments for both Theorem~\ref{thm:k-sub-static} and Theorem~\ref{minkhesssemtraco}. Let
	\begin{equation*}
		\mathbb{S}_\theta^n
		=
		\{x=(x_1,\ldots,x_{n+1})\in\mathbb{S}^n:x_{n+1}\geq\sin\theta\},
	\end{equation*}
	where $\theta\in\left(0,\dfrac{\pi}{2}\right)$. Equivalently, $\mathbb{S}_\theta^n$ is the geodesic ball of radius $\dfrac{\pi}{2}-\theta$ centered at the north pole of the unit sphere. With this notation, we can state the next result.

	\begin{theorem}\cite[Theorem~1.6]{almaraz}\label{thm:almaraz-barbosa-obata}
		Let $(M^n,g)$ be a compact connected Riemannian manifold with a smooth, nonempty boundary $\partial M$. Let $H$ denote the mean curvature of $\partial M$ with respect to the outward unit normal. Suppose that, for some
		$\theta\in(0,\pi/2)$,
		\begin{equation*}
			H\geq(n-1)\tan\theta.
		\end{equation*}
		Then, $(M^n,g)$ is isometric to the standard spherical cap $\mathbb{S}_\theta^n$ if and only if there exists a nonconstant function
		$\phi\in C^\infty(\overline M)$ satisfying
		\begin{equation*}
			\begin{cases}
				\nabla^2\phi+\phi g=0 & \text{in }M,\\
				\tan\theta\,\dfrac{\partial \phi}{\partial\nu}+\phi=0
				& \text{on }\partial M.
			\end{cases}
		\end{equation*}
	\end{theorem}

	\section{Proof of the Main Results}

	\begin{proof}[Proof of Theorem~\ref{theoBetazero}]
		Recall the Bochner formula:
		\begin{equation*}
			\dfrac{1}{2}\Delta|\nabla f|^2=|\nabla^2f|^2+\operatorname{Ric}(\nabla f,\nabla f)+\langle\nabla\Delta f,\nabla f\rangle.
		\end{equation*}
		Equivalently,
		\begin{equation*}
			\dfrac{1}{2}\Delta|\nabla f|^2\geq|\mathring{\nabla}^2f|^2+\dfrac{(\Delta f)^2}{n}+\left(\dfrac{(n-1)\lambda}{n}-\lambda\right)|\nabla f|^2.
		\end{equation*}
		where we used that $|\nabla^2f|^2=|\mathring{\nabla}^2f|^2+\dfrac{(\Delta f)^2}{n}$, $\operatorname{Ric}(\nabla f,\nabla f)\geq\dfrac{(n-1)\lambda}{n}|\nabla f|^2$, and $\Delta f=-\lambda f-\beta$. Hence,
		\begin{equation*}
			\dfrac{1}{2}\Delta|\nabla f|^2\geq|\mathring{\nabla}^2f|^2-\dfrac{(\lambda f+\beta)\Delta f}{n}+\left(\dfrac{(n-1)\lambda}{n}-\lambda\right)|\nabla f|^2.
		\end{equation*}

		Integrating, we obtain
		\begin{equation*}
			\dfrac{1}{2}\int_M\Delta|\nabla f|^2\,dv\geq\int_M|\mathring{\nabla}^2f|^2\,dv+\dfrac{c\beta}{n}\operatorname{Area}(\partial M)+\left(\dfrac{(n-1)\lambda}{n}-\dfrac{(n-1)\lambda}{n}\right)\int_M|\nabla f|^2\,dv.
		\end{equation*}
		Thus, by Stokes' theorem, we have
		\begin{equation*}
			-\dfrac{1}{2c}\int_{\partial M}\langle\nabla|\nabla f|^2,\nabla f\rangle\,da\geq\int_M|\mathring{\nabla}^2f|^2\,dv+\dfrac{c\beta}{n}\operatorname{Area}(\partial M).
		\end{equation*}
		On the other hand, the decomposition of $\Delta f$ in $\partial M$ is given by
		\begin{equation*}
			\Delta f=\Delta^{\partial M}f+H\langle\nabla f,\nu\rangle+\nabla^2f(\nu,\nu).
		\end{equation*}
		where $\nu=-\dfrac{\nabla f}{|\nabla f|}$, and $H$ is the mean curvature of $\partial M$. Thus,
		\begin{equation*}
			-\beta=-cH+\dfrac{1}{2c^2}\langle\nabla|\nabla f|^2,\nabla f\rangle.
		\end{equation*}
		Equivalently,
		\begin{equation*}
			-\dfrac{1}{2c}\langle\nabla|\nabla f|^2,\nabla f\rangle=c\beta-c^2H.
		\end{equation*}

		Therefore,
		\begin{equation*}
			\dfrac{(n-1)c\beta}{n}\operatorname{Area}(\partial M)\geq c^2\int_{\partial M}H\,da+\int_M|\mathring{\nabla}^2f|^2\,dv+\left(\dfrac{(n-1)\lambda}{n}-\dfrac{(n-1)\lambda}{n}\right)\int_M|\nabla f|^2\,dv.
		\end{equation*}
		Assuming $H\geq0$ and $\beta\leq0$, we conclude that
		\begin{equation*}
			0\geq\dfrac{(n-1)c\beta}{n}\operatorname{Area}(\partial M)\geq c^2\int_{\partial M}H\,da+\int_M|\mathring{\nabla}^2f|^2\,dv\geq0.
		\end{equation*}
		Then, $\mathring{\nabla}^2f=0$, $\beta=0$, and $\partial M$ is a minimal surface. Integrating $\Delta f=-\lambda f$ and using $\partial_\nu f=-c$, we obtain
		\begin{equation*}
			\lambda\int_Mf\,dv=c\operatorname{Area}(\partial M)>0.
		\end{equation*}
		Therefore, $\lambda>0$.

		Now, consider the equation
		\begin{equation*}
			\operatorname{div}(\operatorname{Ric}(\nabla f))=\operatorname{div}(\operatorname{Ric})(\nabla f)+\langle\operatorname{Ric},\nabla^2f\rangle=\dfrac{1}{2}\langle\nabla R,\nabla f\rangle+\left\langle\operatorname{Ric},\dfrac{\Delta f}{n}g\right\rangle.
		\end{equation*}
		where we used the twice-contracted second Bianchi identity. Thus, integrating,
		\begin{multline*}
			\int_M\operatorname{div}(\operatorname{Ric}(\nabla f))\,dv=\dfrac{1}{2}\int_M\langle\nabla R,\nabla f\rangle\,dv+\int_M\dfrac{R\Delta f}{n}\,dv \\
			=-\dfrac{n-2}{2n}\int_MR\Delta f\,dv-\dfrac{c}{2}\int_{\partial M}R\,da.
		\end{multline*}
		Equivalently,
		\begin{equation*}
			\int_{\partial M}\operatorname{Ric}\left(\nabla f,-\dfrac{\nabla f}{|\nabla f|}\right)\,da=-\dfrac{n-2}{2n}\int_MR\Delta f\,dv-\dfrac{c}{2}\int_{\partial M}R\,da.
		\end{equation*}
		So,
		\begin{equation*}
			-c\int_{\partial M}\operatorname{Ric}(\nu,\nu)\,da=-\dfrac{n-2}{2n}\int_MR\Delta f\,dv-\dfrac{c}{2}\int_{\partial M}R\,da.
		\end{equation*}

		From the Gauss equation, i.e.,
		\begin{equation*}
			\operatorname{Ric}(\nu,\nu)=\dfrac{R}{2}-\dfrac{R^{\partial M}}{2}+\dfrac{1}{2}H^2-\dfrac{1}{2}|A|^2.
		\end{equation*}
		where $A$ and $H$ stand for the second fundamental form and the mean curvature of $\partial M$ with respect to the induced metric $g$. Hence,
		\begin{equation*}
			c\int_{\partial M}\left(\dfrac{R^{\partial M}}{2}+\dfrac{1}{2}|A|^2\right)\,da=-\dfrac{n-2}{2n}\int_MR\Delta f\,dv.
		\end{equation*}
		where we used that $\partial M$ is a minimal surface. Since $M$ is compact, $R\leq\max_MR$. Moreover, $\Delta f=-\lambda f\leq0$. Therefore,
		\begin{equation*}
			c\int_{\partial M}\left(\dfrac{R^{\partial M}}{2}+\dfrac{1}{2}|A|^2\right)\,da=-\dfrac{n-2}{2n}\int_MR\Delta f\,dv\leq\dfrac{c(n-2)}{2n}\max_M R\,\operatorname{Area}(\partial M).
		\end{equation*}
		Therefore,
		\begin{equation*}
			\int_{\partial M}\left(\dfrac{R^{\partial M}}{2}+\dfrac{1}{2}|A|^2\right)\,da\leq\dfrac{(n-2)R_0}{2n}\operatorname{Area}(\partial M).
		\end{equation*}
		Setting $n=3$,
		\begin{equation*}
			12\pi\mathcal{X}(\partial M)\leq R_0\operatorname{Area}(\partial M).
		\end{equation*}
		where we used the Gauss--Bonnet theorem. Here, $R_0=\max_M R$. Moreover, $\mathcal{X}(\partial M)$ stands for the Euler characteristic of $\partial M$. Now, if equality holds in the above inequality, we will have $A=0$, $\Delta f=-\lambda f$, and $\mathring{\nabla}^2f=0$. Consequently, we may then invoke Reilly's generalization of Obata's theorem (see \cite[Lemma 3]{reilly2}) to conclude that $(M^n,g,f)$ is isometric to the de Sitter system.

	\end{proof}

	\begin{proof}[Proof of Theorem~\ref{prop11}]
		Setting $U=f$, we obtain
		\begin{multline}
			\int_Mf[(\Delta f+knf)^2-|\nabla^2f+kfg|^2]\,dv=(n-1)k\int_M(\Delta f+nkf)f^2\,dv\\
			+\int_M(\nabla^2f-(\Delta f)g-2(n-1)kfg+f\operatorname{Ric})(\nabla f,\nabla f)\,dv,\nonumber
		\end{multline}
		where we used that $\partial M=f^{-1}(0)$. In \cite{li}, the sub-static inequality provides a relation between $\nabla^2f$ and $\operatorname{Ric}$. Since we do not have such a structure, we need another condition on the Ricci tensor.

		Note that
		\begin{multline}\label{reilly007}
			\int_Mf[(\Delta f+knf)^2-|\nabla^2f+kfg|^2]\,dv=(n-1)k\int_M(\Delta f+nkf)f^2\,dv\\
			+\int_M\left(\dfrac{1}{2}\langle\nabla|\nabla f|^2,\nabla f\rangle-(\Delta f)|\nabla f|^2-2(n-1)kf|\nabla f|^2+f\operatorname{Ric}(\nabla f,\nabla f)\right)\,dv.
		\end{multline}
		The Ricci curvature is bounded from below, i.e.,
		\begin{equation*}
			\operatorname{Ric}\geq Kg,\quad K\in\mathbb{R},
		\end{equation*}
		and \eqref{reilly007} yields
		\begin{multline*}
			\int_Mf[(\Delta f+knf)^2-|\nabla^2f+kfg|^2]\,dv\geq(n-1)k\int_M(\Delta f+nkf)f^2\,dv\\
			+\int_M\left(\dfrac{1}{2}\langle\nabla|\nabla f|^2,\nabla f\rangle-(\Delta f)|\nabla f|^2-2(n-1)kf|\nabla f|^2+Kf|\nabla f|^2\right)\,dv.
		\end{multline*}
		Integrating, we obtain
		\begin{multline*}
			\dfrac{n-1}{n}\int_Mf(\Delta f)^2\,dv\geq\int_Mf|\mathring{\nabla}^2f|^2\,dv\\
			+\int_M\left(\dfrac{1}{2}\langle\nabla|\nabla f|^2,\nabla f\rangle-(\Delta f)|\nabla f|^2+Kf|\nabla f|^2\right)\,dv.
		\end{multline*}
		where we used that $|\nabla^2f|^2=|\mathring{\nabla}^2f|^2+\dfrac{(\Delta f)^2}{n}$.

		Using Stokes' theorem, we obtain
		\begin{equation*}
			\dfrac{1}{2}\int_M\langle\nabla|\nabla f|^2,\nabla f\rangle\,dv=-\dfrac{1}{2}\int_Mf\Delta|\nabla f|^2\,dv.
		\end{equation*}
		where we used that $\partial M=f^{-1}(0)$. Thus,
		\begin{equation*}
			\dfrac{n-1}{n}\int_Mf(\Delta f)^2\,dv\geq\int_Mf|\mathring{\nabla}^2f|^2\,dv-\int_M\left(\dfrac{1}{2}f\Delta|\nabla f|^2+(\Delta f-Kf)|\nabla f|^2\right)\,dv.
		\end{equation*}

		Now, we use the formula
		\begin{equation*}
			\operatorname{div}(f\nabla|\nabla f|^2-|\nabla f|^2\nabla f)=f\Delta|\nabla f|^2-\Delta f|\nabla f|^2.
		\end{equation*}
		Integrating this expression over $M$ and applying Stokes' theorem, we obtain
		\begin{equation*}
			\dfrac{1}{2}\int_Mf\Delta|\nabla f|^2\,dv=\dfrac{c^3}{2}\operatorname{Area}(\partial M)+\dfrac{1}{2}\int_M\Delta f|\nabla f|^2\,dv.
		\end{equation*}
		Hence,
		\begin{equation*}
			\dfrac{c^3}{2}\operatorname{Area}(\partial M)+\dfrac{n-1}{n}\int_Mf(\Delta f)^2\,dv\geq\int_Mf|\mathring{\nabla}^2f|^2\,dv-\int_M\left(\dfrac{3}{2}\Delta f-Kf\right)|\nabla f|^2\,dv.
		\end{equation*}
		On the other hand, using \eqref{Serrin}, we obtain
		\begin{equation*}
			\int_Mf(\Delta f)^2\,dv=-\int_M(\lambda f+\beta)f\Delta f\,dv=\int_M(2\lambda f+\beta)|\nabla f|^2\,dv.
		\end{equation*}
		Therefore,
		\begin{multline*}
			\dfrac{c^3}{2}\operatorname{Area}(\partial M)\geq\int_Mf|\mathring{\nabla}^2f|^2\,dv\\
			+\int_M\left(\dfrac{3}{2}(\lambda f+\beta)+Kf\right)|\nabla f|^2\,dv-\dfrac{n-1}{n}\int_M(2\lambda f+\beta)|\nabla f|^2\,dv,
		\end{multline*}
		Equivalently,
		\begin{equation*}
			\dfrac{c^3}{2}\operatorname{Area}(\partial M)\geq\int_Mf|\mathring{\nabla}^2f|^2\,dv+\dfrac{n+2}{2n}\int_M\left[\dfrac{2nK-(n-4)\lambda}{n+2}f+\beta\right]|\nabla f|^2\,dv,
		\end{equation*}
		where equality holds if and only if $\operatorname{Ric}(\nabla f,\nabla f)=K |\nabla f|^{2}$. Take $K=\dfrac{n-1}{n}\lambda$ (and this is the only possible choice) to obtain
		\begin{equation}\label{prin}
			c^3\operatorname{Area}(\partial M)\geq\dfrac{n+2}{n}\int_M(\lambda f+\beta)|\nabla f|^2\,dv.
		\end{equation}
		Furthermore, by H\"older's inequality, we have
		\begin{multline}\label{gradholder}
			\dfrac{1}{2}\left(\int_M(\lambda f+\beta)\Delta f\,dv\right)^2=\dfrac{1}{2}\left(\int_M\sqrt{\varepsilon(-\Delta f)}\left[(\lambda f+\beta)\sqrt{\varepsilon(-\Delta f)}\right]\,dv\right)^2\\
			\leq\dfrac{\varepsilon^2}{2}\int_M\Delta f\,dv\int_M(\lambda f+\beta)^2\Delta f\,dv=-\dfrac{c}{2}\operatorname{Area}(\partial M)\int_M(\lambda f+\beta)^2\Delta f\,dv\\
			=c\lambda\operatorname{Area}(\partial M)\int_M(\lambda f+\beta)|\nabla f|^2\,dv+\dfrac{c^2\beta^2}{2}\operatorname{Area}(\partial M)^2,
		\end{multline}
		where either $\varepsilon=1$, if $\lambda\geq0$ and $\beta\geq0$, or $\varepsilon=-1$, if $\lambda\leq0$ and $\beta\leq0$. Equality holds if and only if $\lambda=0$.
		On the other hand,
		\begin{equation*}
			c^2\operatorname{Area}(\partial M)^2=\left(\int_M\Delta f\,dv\right)^2\leq\operatorname{Vol}(M)\int_M(\Delta f)^2\,dv.
		\end{equation*}
		Equality holds if and only if $\lambda=0$. Therefore,
		\begin{equation*}
			\dfrac{c^4\operatorname{Area}(\partial M)^4}{2\operatorname{Vol}(M)^2}\leq\dfrac{1}{2}\left(\int_M(\lambda f+\beta)\Delta f\,dv\right)^2.
		\end{equation*}
		Hence, combining the above inequality and \eqref{gradholder}, we obtain
		\begin{equation*}
			\dfrac{c}{2}\left(\dfrac{c^2\operatorname{Area}(\partial M)^2}{\operatorname{Vol}(M)^2}-\beta^2\right)\operatorname{Area}(\partial M)\leq\lambda\int_M(\lambda f+\beta)|\nabla f|^2\,dv.
		\end{equation*}
		Combining the above inequality with \eqref{prin} and $\lambda\geq0$, we obtain the result. Indeed,
		\begin{equation*}
			\operatorname{Area}(\partial M)^2\leq\left(\dfrac{2n\lambda}{n+2}+\dfrac{\beta^2}{c^2}\right)\operatorname{Vol}(M)^2,
		\end{equation*}
		where equality holds if and only if $\mathring{\nabla}^2f=0$, $\lambda=0$, and $\operatorname{Ric}(\nabla f,\nabla f)=0$. Then, from Lemma 6 in \cite{farina}, $M$ is a metric ball and $f$ is a radial function, i.e., $f$ depends only on the distance from the center of the ball; see \cite{serrin,Weinberger}. Assuming that $f$ and $g$ extend smoothly to $\partial M$, equality holds if and only if $(M^n,g,f)$ is a compact geodesic ball $B_{r_0}(p)$ and $f(r)=\dfrac{\beta}{2n}(r_0^2-r^2)$, where $r=d(p,\cdot)$ is the distance function from a point $p\in M$ satisfying $\nabla f(p)=0$.
		Moreover, $(M,g)$ admits the global warped product decomposition
		\begin{equation*}
			g= dr^2+\psi(r)^2g_{\mathbb{S}^{n-1}}.
		\end{equation*}
		Here, $\psi(r)=f'(r)=\dfrac{\beta}{n}r$, and $r_0=d(p,q)$, where $q\in\partial M$. Recall the discussion of concircular fields in Section~\ref{back}.

		In fact, we have
		\begin{equation*}
			\nabla^2f=\phi g,\qquad \phi=-\dfrac{\beta}{n},
		\end{equation*}
		so that $f$ is a concircular scalar field with a constant characteristic function.
		By Corollary~\ref{cor:geodesicball}, there exists a unique critical point $p\in M$ such that
		\begin{equation*}
			M=B_{r_0}(p)\qquad\text{and}\qquad \partial M=S_{r_0}(p).
		\end{equation*}
		Furthermore, Corollary~\ref{cor:quadratic-potential} shows that
		\begin{equation*}
			f(r)={\dfrac{\beta}{2n}}r_0^2-\dfrac{\beta}{2n}r^2,
		\end{equation*}
		where $r=d(p,\cdot)$ is the distance function from $p$. Hence, $f$ depends only on the geodesic distance from $p$, and the integral curves of $\nabla f$ are precisely the radial geodesics starting at $p$. Therefore, every point of $M\setminus\{p\}$ is regular for $f$, and Theorem~\ref{lem:tashiro} applies on $M\setminus\{p\}$. Consequently, the radial coordinate in Theorem~\ref{lem:tashiro} coincides with the geodesic distance from $p$ up to orientation (see Remark~\ref{rem:tashiro-orientation}), and the metric admits a global warped product decomposition
		\begin{align*}
			g=dr^2+\psi(r)^2g_{\mathbb{S}^{n-1}},\qquad 0\leq r\leq r_0,
		\end{align*}
		which immediately yields $\psi(r) = f'(r)$; see also \cite{tashiro}.

	\end{proof}

	\begin{proof}[Proof of Theorem~\ref{divkato}]
		Consider an orthonormal frame $\{e_1,\ldots,e_n\}$ such that $e_1=\dfrac{\nabla f}{|\nabla f|}$. Therefore, from the Cauchy--Schwarz inequality \cite[page 277]{petersen}, we have
		\begin{align*}
			|\nabla^2f|^2 = f_{,11}^2 +2\sum_{i=2}^{n}f_{,1i}^2 + \sum_{i,j=2}^{n}f_{,ij}^2
			& \geq  f_{,11}^2+2\sum_{i=2}^{n}f_{,1i}^2+\dfrac{1}{(n-1)}\left(\sum_{i=2}^{n}f_{,ii}\right)^2\\
			&= f_{,11}^2+2\sum_{i=2}^{n}f_{,1i}^2+\dfrac{1}{n-1}(\Delta f-f_{,11})^2\\
			&=\dfrac{n}{n-1}f_{,11}^2+2\sum_{i=2}^{n}f_{,1i}^2+\dfrac{\Delta f}{n-1}(\Delta f-2f_{,11})\\
			&\geq\dfrac{n}{n-1}f_{,11}^2+\dfrac{n}{n-1}\sum_{i=2}^{n}f_{,1i}^2+\dfrac{\Delta f}{n-1}(\Delta f-2f_{,11})\\
			&=\dfrac{n}{n-1}|\nabla|\nabla f||^2+\dfrac{\Delta f}{n-1}(\Delta f-2f_{,11}),
		\end{align*}
		where we used that $f_{,ij} = \nabla^2f(e_i,e_j)$ and $2\geq\dfrac{n}{(n-1)}$.

		Since
		\begin{align*}
			f_{,11}=\nabla^2f\left(\dfrac{\nabla f}{|\nabla f|},\dfrac{\nabla f}{|\nabla f|}\right)=\dfrac{1}{2|\nabla f|^2}\langle\nabla|\nabla f|^2,\nabla f\rangle,
		\end{align*}
		Equivalently,
		\begin{align*}
			|\nabla^2f|^2\geq\dfrac{n}{n-1}|\nabla|\nabla f||^2+\dfrac{\Delta f}{n-1}\left(\Delta f-\dfrac{1}{|\nabla f|^2}\langle\nabla|\nabla f|^2,\nabla f\rangle\right).
		\end{align*}

		Now, using that $|\mathring{\nabla}^2f|^{2} = |\nabla^2f|^{2} - \dfrac{(\Delta f)^{2}}{n}$ and
		\begin{align*}
			\dfrac{1}{|\nabla f|^2}\left|\dfrac{1}{2}\nabla|\nabla f|^2-\dfrac{\Delta f}{n}\nabla f\right|^2=|\nabla|\nabla f||^2-\dfrac{\Delta f}{n|\nabla f|^2}\langle\nabla|\nabla f|^2,\nabla f\rangle+\left(\dfrac{\Delta f}{n}\right)^2.
		\end{align*}
		Therefore,
		\begin{align*}
			|\mathring{\nabla}^2f|^2\geq\dfrac{n}{n-1}\dfrac{1}{|\nabla f|^2}\left|\dfrac{1}{2}\nabla|\nabla f|^2-\dfrac{\Delta f}{n}\nabla f\right|^2.
		\end{align*}

		Equality holds if and only if
		\begin{align}\label{sos}
			|\mathring{\nabla}^2f|^2 = \dfrac{n}{n-1}\dfrac{1}{|\nabla f|^2}|\mathring{\nabla}^2f(\nabla f)|^2.
		\end{align}
		Equivalently,
		\begin{equation*}
			\mathring{\nabla}^2f=\dfrac{n\mu}{n-1}\dfrac{df\otimes df}{|\nabla f|^{2}} - \dfrac{\mu}{n-1}g,
		\end{equation*}
		where
		\begin{equation*}
			\mathring{\nabla}^2f=\operatorname{diag}(\mu,-\dfrac{\mu}{n-1},\ldots,-\dfrac{\mu}{n-1}).
		\end{equation*}

		Then, from the classical Bochner formula
		\begin{align*}
			\dfrac{1}{2}\Delta|\nabla f|^2=|\mathring{\nabla}^2f|^2+\operatorname{Ric}(\nabla f,\nabla f)+\langle\nabla\Delta f,\nabla f\rangle+\dfrac{(\Delta f)^2}{n}.
		\end{align*}
		Thus,
		\begin{align*}
			\dfrac{1}{2}\Delta|\nabla f|^2-\operatorname{Ric}(\nabla f,\nabla f)-\langle\nabla\Delta f,\nabla f\rangle-\dfrac{(\Delta f)^2}{n}\geq\dfrac{n}{n-1}\dfrac{1}{|\nabla f|^2}\left|\dfrac{1}{2}\nabla|\nabla f|^2-\dfrac{\Delta f}{n}\nabla f\right|^2.
		\end{align*}
		Note that
		\begin{multline*}
			\operatorname{div}\left(\dfrac{1}{2f}\nabla|\nabla f|^2-\dfrac{\Delta f}{nf}\nabla f\right)=\dfrac{1}{2f}\Delta|\nabla f|^2-\dfrac{1}{2f^2}\langle\nabla|\nabla f|^2,\nabla f\rangle\\
			-\dfrac{1}{nf}\langle\nabla\Delta f,\nabla f\rangle-\dfrac{(\Delta f)^2}{nf}+\dfrac{\Delta f}{nf^2}|\nabla f|^2,
		\end{multline*}
		Equivalently,
		\begin{multline*}
			f\operatorname{div}\left(\dfrac{1}{2f}\nabla|\nabla f|^2-\dfrac{\Delta f}{nf}\nabla f\right)=\dfrac{1}{2}\Delta|\nabla f|^2-\dfrac{1}{2f}\langle\nabla|\nabla f|^2,\nabla f\rangle\\
			-\dfrac{1}{n}\langle\nabla\Delta f,\nabla f\rangle-\dfrac{(\Delta f)^2}{n}+\dfrac{\Delta f}{nf}|\nabla f|^2,
		\end{multline*}

		Therefore,
		\begin{multline*}
			f\operatorname{div}\left(\dfrac{1}{2f}\nabla|\nabla f|^2-\dfrac{\Delta f}{nf}\nabla f\right)\\
			-\dfrac{1}{f}\left(f\operatorname{Ric}(\nabla f,\nabla f)-\dfrac{1}{2}\langle\nabla|\nabla f|^2,\nabla f\rangle+\dfrac{f(n-1)}{n}\langle\nabla\Delta f,\nabla f\rangle+\dfrac{\Delta f}{n}|\nabla f|^2\right)\\
			\geq\dfrac{n}{n-1}\dfrac{1}{|\nabla f|^2}\left|\dfrac{1}{2}\nabla|\nabla f|^2-\dfrac{\Delta f}{n}\nabla f\right|^2.
		\end{multline*}
		Equivalently,
		\begin{multline*}
			f\operatorname{div}\left(\dfrac{1}{2f}\nabla|\nabla f|^2-\dfrac{\Delta f}{nf}\nabla f\right)
			\geq\dfrac{n}{n-1}\dfrac{1}{|\nabla f|^2}\left|\dfrac{1}{2}\nabla|\nabla f|^2-\dfrac{\Delta f}{n}\nabla f\right|^2\\
			+\dfrac{1}{f}\left(f\operatorname{Ric}(\nabla f,\nabla f)-\nabla^2f(\nabla f,\nabla f)+\dfrac{f(n-1)}{n}\langle\nabla\Delta f,\nabla f\rangle+\dfrac{\Delta f}{n}|\nabla f|^2\right).
		\end{multline*}

		Using
		\begin{align*}
			\nabla\left(\dfrac{\Delta f}{f}\right)=\dfrac{1}{f^2}(f\nabla\Delta f-\Delta f\nabla f)
		\end{align*}
		we obtain the result, namely,
		\begin{multline*}
			f\operatorname{div}\left(\dfrac{1}{2f}\nabla|\nabla f|^2-\dfrac{\Delta f}{nf}\nabla f\right)
			\geq \dfrac{n}{n-1}\dfrac{1}{|\nabla f|^2}\left|\dfrac{1}{2}\nabla|\nabla f|^2-\dfrac{\Delta f}{n}\nabla f\right|^2\\
			-\dfrac{1}{f}\mathfrak{L}_g^{\ast}(f)(\nabla f,\nabla f)
			+\dfrac{n-1}{n}f\left\langle\nabla\left(\dfrac{\Delta f}{f}\right),\nabla f\right\rangle.
		\end{multline*}

		Moreover, since
		\begin{align*}
			f\operatorname{div}\left(\dfrac{1}{2f}\nabla|\nabla f|^2-\dfrac{\Delta f}{nf}\nabla f\right)
			+ \dfrac{1}{f}\mathfrak{L}_g^{\ast}(f)(\nabla f,\nabla f)
			- \dfrac{n-1}{n}f\left\langle\nabla\left(\dfrac{\Delta f}{f}\right),\nabla f\right\rangle = |\mathring{\nabla}^2f|^{2},
		\end{align*}
		we have $\mathring{\nabla}^2f=0$ if and only if
		\begin{align*}
			f\operatorname{div}\left(\dfrac{1}{2f}\nabla|\nabla f|^2-\dfrac{\Delta f}{nf}\nabla f\right)
			+ \dfrac{1}{f}\mathfrak{L}_g^{\ast}(f)(\nabla f,\nabla f)
			- \dfrac{n-1}{n}f\left\langle\nabla\left(\dfrac{\Delta f}{f}\right),\nabla f\right\rangle = 0.
		\end{align*}

	\end{proof}

	\begin{proof}[Proof of Theorem~\ref{sabo}]

		From \eqref{divelectrostatic}, we have
		\begin{multline*}
			\int_M\operatorname{div}\left(\operatorname{Ric}(\nabla f)+\dfrac{2}{n}((n-2)|E|^2-\lambda)\nabla f\right)\,dv\\
			-2\int_M\operatorname{div}((|E|^2g-E^{\flat}\otimes E^{\flat})(\nabla f))\,dv\\
			\geq\dfrac{2(n-2)}{n}\int_M\langle\nabla|E|^2,\nabla f\rangle\,dv.
		\end{multline*}

		By Stokes' theorem, we obtain
		\begin{multline*}
			-\int_{\partial M}\left\langle\operatorname{Ric}(\nabla f)+\dfrac{2}{n}((n-2)|E|^2-\lambda)\nabla f,\dfrac{\nabla f}{|\nabla f|}\right\rangle\,da\\
			+2\int_{\partial M}(|E|^2g-E^\flat\otimes E^\flat)\left(\nabla f,\dfrac{\nabla f}{|\nabla f|}\right)\,da\\
			\geq\dfrac{2(n-2)}{n}\int_M\langle\nabla|E|^2,\nabla f\rangle\,dv.
		\end{multline*}
		It is well known that $E$ and $\nabla f$ are parallel along the boundary $\partial M=f^{-1}(0)$, i.e.,
		\begin{align*}
			\langle E,\nabla f\rangle^2=|E|^2|\nabla f|^2,
		\end{align*}
		see Proposition~\ref{properties}. Thus,
		\begin{multline*}
			-\int_{\partial M}\left\langle\operatorname{Ric}(\nabla f)+\dfrac{2}{n}((n-2)|E|^2-\lambda)\nabla f,\dfrac{\nabla f}{|\nabla f|}\right\rangle\,da\\
			\geq-\dfrac{2(n-2)}{n}\int_M|E|^2\Delta f\,dv-\dfrac{2c(n-2)}{n}\int_{\partial M}|E|^2\,da.
		\end{multline*}
		Equivalently,
		\begin{align*}
			\dfrac{2c\lambda}{n}\operatorname{Area}(\partial M)\geq c\int_{\partial M}\operatorname{Ric}(\nu,\nu)\,da-\dfrac{2(n-2)}{n}\int_M|E|^2\Delta f\,dv.
		\end{align*}

		The Gauss equation gives
		\begin{align*}
			\operatorname{Ric}(\nu,\nu)=\dfrac{R}{2}-\dfrac{R^{\partial M}}{2}+\dfrac{1}{2}H^2-\dfrac{1}{2}|A|^2.
		\end{align*}
		Here, $A$ and $H$ stand for the second fundamental form and the mean curvature of $\partial M$ with respect to the induced metric $g$. Hence, if $\partial M$ is totally geodesic,
		\begin{align*}
			\dfrac{2c\lambda}{n}\operatorname{Area}(\partial M)\geq c\int_{\partial M}\left(\dfrac{R}{2}-\dfrac{R^{\partial M}}{2}\right)\,da-\dfrac{2(n-2)}{n}\int_M|E|^2\Delta f\,dv.
		\end{align*}
		Since $R=2(|E|^2+\lambda)$,
		\begin{align*}
			c\int_{\partial M}\dfrac{R^{\partial M}}{2}\,da\geq\dfrac{c(n-2)\lambda}{n}\operatorname{Area}(\partial M)+c\int_{\partial M}|E|^2\,da-\dfrac{2(n-2)}{n}\int_M|E|^2\Delta f\,dv.
		\end{align*}

		Thus, from \eqref{QE}, we have
		\begin{multline*}
			c\int_{\partial M}\dfrac{R^{\partial M}}{2}\,da
			\geq\dfrac{c(n-2)\lambda}{n}\operatorname{Area}(\partial M)\\
			+\dfrac{c\omega_{n-1}^2(n-1)(n-2)}{2\operatorname{Area}(\partial M)}Q(\partial M)^2
			-\dfrac{2(n-2)}{n}\int_M|E|^2\Delta f\,dv.
		\end{multline*}

		On the other hand, by \cite[Proposition~9.2.2]{petersen}, if $\operatorname{div}E = 0$, then
		\begin{align*}
			\dfrac{1}{2}\Delta|E|^{2} = |\nabla E|^{2} + \underbrace{\nabla_E(\operatorname{div}E)}_{=0} + \operatorname{Ric}(E,E).
		\end{align*}
		Green's identity on $M$ gives
		\begin{align*}
			\int_{M}\bigl(f\Delta|E|^2-|E|^2\Delta f\bigr)\,dv 
			= \int_{\partial M}\bigl(f\partial_\nu|E|^2-|E|^2\partial_\nu f\bigr)\,da.
		\end{align*}
		Since $f=0$ and $\partial_\nu f=-c$ on $\partial M$, we have
		\begin{align*}
			-\int_{M}|E|^2\Delta f\,dv 
			= -2\int_{M}f\bigl(|\nabla E|^2+\operatorname{Ric}(E,E)\bigr)\,dv+c\int_{\partial M}|E|^2\,da.
		\end{align*}
		Applying \eqref{QE}, we obtain
		\begin{equation*}
			-\int_M|E|^2\Delta f\,dv \geq -2\int_Mf\bigl(|\nabla E|^2+\operatorname{Ric}(E,E)\bigr)\,dv+\dfrac{c\omega_{n-1}^2(n-1)(n-2)}{2\operatorname{Area}(\partial M)}Q(\partial M)^2.
		\end{equation*}
		Multiplying the preceding inequality by $\dfrac{2(n-2)}{n}$, we conclude
		\begin{multline*}
			-\dfrac{2(n-2)}{n}\int_M|E|^2\Delta f\,dv \\
			\geq -\dfrac{4(n-2)}{n}\int_Mf\bigl(|\nabla E|^2+\operatorname{Ric}(E,E)\bigr)\,dv \\
			+ \dfrac{2(n-2)c\omega_{n-1}^2(n-1)(n-2)}{2n\operatorname{Area}(\partial M)}Q(\partial M)^2.
		\end{multline*}
		Equality holds if and only if $E$ is parallel to $\nu$ and $\mathring{\nabla}^2f=0$. Thus,
		\begin{multline*}
			c\int_{\partial M}\dfrac{R^{\partial M}}{2}\,da
			\geq\dfrac{c(n-2)\lambda}{n}\operatorname{Area}(\partial M)
			+\dfrac{c\omega_{n-1}^2(n-1)(n-2)}{2\operatorname{Area}(\partial M)}Q(\partial M)^2\\
            +\dfrac{(n-2)c\omega_{n-1}^2(n-1)(n-2)}{n\operatorname{Area}(\partial M)}Q(\partial M)^2
            -\dfrac{4(n-2)}{n}\int_Mf\bigl(|\nabla E|^2+\operatorname{Ric}(E,E)\bigr)\,dv .
		\end{multline*}

        	The term $|E|^{2}\Delta f$ does not provide much information. To handle this term, we assumed $\operatorname{div}E=0$. The above inequality eliminates $|E|^{2}\Delta f$ and provides information about the curvature of the manifold in the direction of the electric vector field $E$, offering a geometric interpretation for the ``bad'' term.

%In addition, assuming that $E$ is a Killing vector field in the above corollary, the condition $\operatorname{Ric}(E,E) +\langle\operatorname{tr}\nabla^2E,\,E\rangle) = 0$; see \cite[page 230]{petersen}.

Now, if equality holds in the above inequality, we will have $E$ parallel to $\nu = - \dfrac{\nabla f}{\vert\nabla f\vert}$, and $\mathring{\nabla}^2f=0$. Consequently, the scalar curvature $R$ is constant \cite[Proposition 3]{freitas} and since $\partial M$ is totally geodesic, we may then invoke Reilly's generalization of Obata's theorem (see \cite[Lemma 3]{reilly2}) to conclude that $(M^n,g,f)$ is isometric to the de Sitter system.
	\end{proof}

    %%%%%%%%%%%%%%%%%%%%%%%%%%%%%%%%%%%%%%%%%%%%%%%%%%%%%%%%%%%%%%%%%%%%%%%%%%%%%%%%%%%%%%%%%%%%%%%%%%%%%%%%%%%%%%%%%%%%%%%%%%%%%%%%%%%%%%%%%%%%%%%%%%%%%%%%%%%%%%%%%%%%%%%%%%%%%%%%%%%%%%%%%%%%%%%%%%%%%%%%%%%%%%%%%%%%%%%%%%%%%%%%%%%%%%%%%%%%%%%%%%%%%%%%%%%%%%%%%%%%%%%%%%%%%%%%%%%%%%%%%%%%%%%%%%%%%%%%%%%%%%%%%%%%%

	\begin{proof}[Proof of Theorem~\ref{thm:k-sub-static}]
		From Theorem~\ref{divkato}, we obtain
		\begin{multline*}
			\int_M\operatorname{div}\left(\dfrac{1}{2f}\nabla|\nabla f|^2-\dfrac{\Delta f}{nf}\nabla f\right)\,dv
			\geq - \int_M\dfrac{1}{f^{2}} \mathfrak{L}^{\ast}_g(f)(\nabla f,\nabla f)\,dv\\
			+ \dfrac{n-1}{n}\int_M\left\langle\nabla\left(\dfrac{\Delta f}{f}\right),\nabla f\right\rangle\,dv.
		\end{multline*}
		Moreover,
		\begin{align*}
			\int_M\operatorname{div}\left(\dfrac{1}{2f}\nabla|\nabla f|^2-\dfrac{\Delta f}{nf}\nabla f\right)\,dv
			=-\dfrac{1}{c\kappa}\int_{\partial M}\left[\dfrac{1}{2}\langle\nabla|\nabla f|^2,\nabla f\rangle-\dfrac{c^2\Delta f}{n}\right]\,da.
		\end{align*}
		The decomposition of $\Delta f$ on $\partial M$ is given by
		\begin{align*}
			\Delta f=\Delta^{\partial M}f+H\langle\nabla f,\nu\rangle+\nabla^2f(\nu,\nu),
		\end{align*}
		Equivalently,
		\begin{align*}
			c^2\Delta f+c^3H=\dfrac{1}{2}\langle\nabla|\nabla f|^2,\nabla f\rangle,
		\end{align*}
		where $\nu=-\dfrac{\nabla f}{|\nabla f|}$ and $|\nabla f|=c$ on $\partial M=f^{-1}(\kappa)$.
		Hence,
		\begin{align}\label{hum}
			\int_M\operatorname{div}\left(\dfrac{1}{2f}\nabla|\nabla f|^2-\dfrac{\Delta f}{nf}\nabla f\right)\,dv=-\dfrac{1}{\kappa}\int_{\partial M}\left[\dfrac{c(n-1)}{n}\Delta f+c^2H\right]\,da.
		\end{align}
		Since $\Delta f=-\lambda\kappa-\beta$ on $\partial M$, we obtain
		\begin{multline*}
			-\dfrac{1}{\kappa}\int_{\partial M}\left[\dfrac{c(n-1)}{n}\Delta f+c^2H\right]\,da
			=\dfrac{c(\lambda\kappa+\beta)(n-1)}{n\kappa}\operatorname{Area}(\partial M)\\
			-\dfrac{c^2}{\kappa}\int_{\partial M}H\,da
			\geq\int_M\dfrac{1}{f^{2}}\left(- \mathfrak{L}^{\ast}_g(f) + \dfrac{\beta(n-1)}{n}g\right)(\nabla f,\nabla f)\,dv.
		\end{multline*}
		Therefore, since the manifold is sub-static and $\beta\geq0$,
		\begin{align*}
			(\lambda\kappa+\beta)\operatorname{Area}(\partial M)\geq nc\int_{\partial M}\dfrac{H}{n-1}\,da,
		\end{align*}
		where
		\begin{align*}
			c\operatorname{Area}(\partial M) = \lambda\int_Mf\,dv + \beta\operatorname{Vol}(M).
		\end{align*}

		If equality holds in the above inequality, then
		\begin{align*}
			0\leq\int_M\dfrac{1}{f^{2}}\left(- \mathfrak{L}^{\ast}_g(f) + \dfrac{\beta(n-1)}{n}g\right)(\nabla f,\nabla f)\,dv=0.
		\end{align*}
		Since $-\mathfrak{L}^{\ast}_g(f)\geq0$ and $\beta\geq0$, it follows that
		\begin{equation*}
			0\geq\dfrac{\beta(n-1)}{n}\int_M\dfrac{|\nabla f|^2}{f^2}\,dv\geq0.
		\end{equation*}
		The last integral is strictly positive. Indeed, $\partial_\nu f=-c\neq0$ on $\partial M$, so continuity implies that $\nabla f$ is nonzero in an interior neighborhood of $\partial M$. Therefore, $\beta=0$.

		Moreover, equality in Theorem~\ref{divkato} yields $\mathring{\nabla}^2f=0$. Since $\beta=0$, the equation becomes $\Delta f=-\lambda f$. In particular, since the manifold is sub-static,
		\begin{align*}
			\operatorname{Ric}\geq\dfrac{(n-1)\lambda}{n}
			g.
		\end{align*}

		Integrating $\Delta f=-\lambda f$ over $M$ and using the divergence theorem together with $\dfrac{\partial f}{\partial\nu}=-c$, we obtain
		\begin{align*}
			-c\,\operatorname{Area}(\partial M)=-\lambda\int_Mf\,dv.
		\end{align*}
		Since $f>0$ on $M$, it follows that
		\begin{align*}
			\lambda=\dfrac{c\,\operatorname{Area}(\partial M)}{\int_Mf\,dv}>0.
		\end{align*}

		To apply \cite[Theorem~1.6]{almaraz}, consider the rescaled metric $\widetilde g=\dfrac{\lambda}{n}g$, which satisfies
		\begin{align*}
			\nabla_{\widetilde g}^2f+f\widetilde g=0.
		\end{align*}
		Defining
		\begin{align*}
			\tan\theta=\dfrac{\kappa}{c}\sqrt{\dfrac{\lambda}{n}},
		\end{align*}
		the boundary conditions become
		\begin{align*}
			\tan\theta\,\dfrac{\partial f}{\partial\widetilde\nu}+f=0,
			\qquad
			\widetilde H\geq(n-1)\tan\theta.
		\end{align*}

		The parametrization in \cite[Theorem~1.6]{almaraz} is given by the angle $\theta\in\left(0,\dfrac{\pi}{2}\right)$, which is related to the geodesic radius of the geodesic ball by
		$r=\dfrac{\pi}{2}-\theta$. Hence,
		$(n-1)\cot r=(n-1)\tan\theta$.
		In our setting, choosing
		$\tan\theta=\dfrac{\kappa}{c}\sqrt{\dfrac{\lambda}{n}},$
		the boundary condition in \cite{almaraz} becomes
		$\tan\theta\,\dfrac{\partial f}{\partial\widetilde\nu}+f=0,$
		and the hypothesis
		$H\geq\dfrac{(n-1)\lambda\kappa}{nc}$
		is equivalent to
		$\widetilde H\geq(n-1)\tan\theta$.

		Therefore, all assumptions of Theorem~\ref{thm:almaraz-barbosa-obata} are satisfied. Hence, $(M,\widetilde g)$ is isometric to a geodesic ball in the unit sphere. Rescaling back to the original metric, we conclude that $(M,g)$ is isometric to a geodesic ball in a sphere of constant sectional curvature $\lambda/n$.
	\end{proof}

	\begin{proof}[Proof of Proposition~\ref{sinal da f na v sub static}]
		Let $(M^n,g,f,\beta)$ be a compact $V$-sub-static manifold without boundary. From Theorem~\ref{divkato} and the $V$-sub-static metric, we have
		\begin{multline*}
			\operatorname{div}\left(\dfrac{1}{2}\nabla|\nabla f|^2-\dfrac{\Delta f}{n}\nabla f\right) - \left\langle \left(\dfrac{1}{2f}\nabla|\nabla f|^2-\dfrac{\Delta f}{nf}\nabla f\right),\nabla f\right\rangle\\
			= f\operatorname{div}\left(\dfrac{1}{2f}\nabla|\nabla f|^2-\dfrac{\Delta f}{nf}\nabla f\right) \geq0.
		\end{multline*}
		Hence,
		\begin{align*}
			0=\int_M\operatorname{div}\left(\dfrac{1}{2}\nabla|\nabla f|^2-\dfrac{\Delta f}{n}\nabla f\right) \geq \int_M\left\langle \left(\dfrac{1}{2f}\nabla|\nabla f|^2-\dfrac{\Delta f}{nf}\nabla f\right),\nabla f\right\rangle\,dv.
		\end{align*}

		On the other hand, by Stokes' theorem, we have
		\begin{multline*}
			\int_M\left\langle \left(\dfrac{1}{2f}\nabla|\nabla f|^2-\dfrac{\Delta f}{nf}\nabla f\right),\nabla f\right\rangle\,dv = - \dfrac{1}{2}\int_M \left(\dfrac{\Delta f}{f} - \dfrac{|\nabla f|^{2}}{f^{2}}\right)|\nabla f|^{2}\,dv\\
			-\dfrac{1}{n}\int_M  \dfrac{\Delta f}{f}|\nabla f|^{2}\,dv,
		\end{multline*}
		Equivalently,
		\begin{align*}
			0\geq \dfrac{(n+2)}{2n(n-1)}\int_M  \left(R + \dfrac{nk}{f}\right)|\nabla f|^{2}\,dv +\dfrac{1}{2}\int_M\dfrac{|\nabla f|^{4}}{f^2}\,dv.
		\end{align*}

		First, suppose that $R=0$. In this case, $(n-1)\Delta f=-nk$; hence, by the maximum principle, $f$ is constant. Now, suppose that $R<0$. We can choose points $p,q\in M$ such that $f(q)=\displaystyle\min_M f$ and $f(p)=\displaystyle\max_M f$. Then,
		\begin{equation*}
			-Rf(q) - kn \leq -Rf - kn \leq -Rf(p) - kn,
		\end{equation*}
		and so
		\begin{equation*}
			0 \leq \Delta f(q) \leq \Delta f \leq \Delta f(p) \leq 0.
		\end{equation*}
		Thus, $\Delta f=0$, and therefore $f$ must be constant. Consequently, $R$ is a positive constant; see \cite[Theorem 7]{miaotam2009}.

		We conclude that
		\begin{align*}
			0\geq \dfrac{(n+2)R}{2n(n-1)}\int_M|\nabla f|^{2}\,dv +\dfrac{k(n+2)}{2(n-1)}\int_M  \dfrac{|\nabla f|^{2}}{f}\,dv +\dfrac{1}{2}\int_M\dfrac{|\nabla f|^{4}}{f^2}\,dv\geq0.
		\end{align*}
	\end{proof}

	\begin{proof}[Proof of Theorem~\ref{minkhesssemtraco}]

		The boundary decomposition of the Laplacian is
		\begin{equation}\label{eq:mink-boundary-laplacian}
			\Delta f=\Delta_{\partial M}f
			+H\dfrac{\partial f}{\partial\nu}
			+\nabla^2f(\nu,\nu).
		\end{equation}
		Moreover,
		\begin{equation*}
			\nabla f=\nabla_{\partial M}f+\dfrac{\partial f}{\partial\nu}\nu
			\quad\text{on }\partial M.
		\end{equation*}
		It follows that
		\begin{align*}
			\dfrac{\partial f}{\partial\nu}\nabla^2f(\nu,\nu)
			&=\nabla^2f(\nabla f,\nu)
			-\nabla^2f(\nabla_{\partial M}f,\nu)\notag\\
			&=\dfrac{1}{2}\langle\nabla|\nabla f|^2,\nu\rangle
			-\nabla^2f(\nabla_{\partial M}f,\nu).
			\label{eq:mink-normal-hessian}
		\end{align*}
		For every vector field $X$ tangent to $\partial M$, the Weingarten formula
		gives
		\begin{equation*}
			\nabla^2f(X,\nu)
			=\left\langle
			\nabla_{\partial M}\left(\dfrac{\partial f}{\partial\nu}\right),X
			\right\rangle
			-A(X,\nabla_{\partial M}f).
		\end{equation*}
		Multiplying \eqref{eq:mink-boundary-laplacian} by
		$f^{-1}\dfrac{\partial f}{\partial\nu}$, we obtain
		\begin{align*}
			\dfrac{\Delta f}{f}\dfrac{\partial f}{\partial\nu}
			&=\dfrac{1}{f}\dfrac{\partial f}{\partial\nu}\Delta_{\partial M}f
			+\dfrac{H}{f}\left(\dfrac{\partial f}{\partial\nu}\right)^2
			+\dfrac{1}{2f}\langle\nabla|\nabla f|^2,\nu\rangle\notag\\
			&\quad-\dfrac{1}{f}
			\left\langle
			\nabla_{\partial M}\left(\dfrac{\partial f}{\partial\nu}\right),
			\nabla_{\partial M}f
			\right\rangle
			+\dfrac{1}{f}A(\nabla_{\partial M}f,\nabla_{\partial M}f).
		\end{align*}

		Integration by parts gives
		\begin{align*}
			\int_{\partial M}\dfrac{1}{f}\dfrac{\partial f}{\partial\nu}
			\Delta_{\partial M}f\,da
			&=-\int_{\partial M}\dfrac{1}{f}
			\left\langle
			\nabla_{\partial M}\left(\dfrac{\partial f}{\partial\nu}\right),
			\nabla_{\partial M}f
			\right\rangle\,da\notag\\
			&\quad+\int_{\partial M}\dfrac{1}{f^2}
			\dfrac{\partial f}{\partial\nu}
			|\nabla_{\partial M}f|^2\,da.
			\label{eq:mink-tangential-ibp}
		\end{align*}
		Consequently,
		\begin{multline*}
			\int_{\partial M}\dfrac{\Delta f}{f}
			\dfrac{\partial f}{\partial\nu}\,da
			=\int_{\partial M}\dfrac{H}{f}
			\left(\dfrac{\partial f}{\partial\nu}\right)^2\,da
			+\int_{\partial M}\dfrac{1}{2f}
			\langle\nabla|\nabla f|^2,\nu\rangle\,da\\
			+\int_{\partial M}\dfrac{1}{f}
			A(\nabla_{\partial M}f,\nabla_{\partial M}f)\,da+\int_{\partial M}\dfrac{1}{f^2}
			\dfrac{\partial f}{\partial\nu}
			|\nabla_{\partial M}f|^2\,da\notag\\
			-2\int_{\partial M}\dfrac{1}{f}
			\left\langle
			\nabla_{\partial M}\left(\dfrac{\partial f}{\partial\nu}\right),
			\nabla_{\partial M}f
			\right\rangle\,da.
		\end{multline*}
		The boundary condition implies
		\begin{equation*}
			\dfrac{\partial f}{\partial\nu}=-cf,
			\qquad
			\nabla_{\partial M}\left(\dfrac{\partial f}{\partial\nu}\right)
			=-c\nabla_{\partial M}f.
		\end{equation*}
		Hence, we obtain
		\begin{multline*}
			\dfrac{1}{f}A(\nabla_{\partial M}f,\nabla_{\partial M}f)
			+\dfrac{1}{f^2}\dfrac{\partial f}{\partial\nu}
			|\nabla_{\partial M}f|^2-\dfrac{2}{f}
			\left\langle
			\nabla_{\partial M}\left(\dfrac{\partial f}{\partial\nu}\right),
			\nabla_{\partial M}f
			\right\rangle \\
			=\dfrac{1}{f}
			[A(\nabla_{\partial M}f,\nabla_{\partial M}f)
			+c|\nabla_{\partial M}f|^2].
		\end{multline*}
		Therefore,
		\begin{multline}
			-c\int_{\partial M}\Delta f\,da
			=c^2\int_{\partial M}fH\,da
			+\int_{\partial M}\dfrac{1}{2f}
			\langle\nabla|\nabla f|^2,\nu\rangle\,da\\
			+\int_{\partial M}\dfrac{1}{f}
			[A(\nabla_{\partial M}f,\nabla_{\partial M}f)
			+c|\nabla_{\partial M}f|^2]\,da.
			\label{eq:mink-boundary-corrected}
		\end{multline}

		Subtracting $1/n$ times the left-hand side of
		\eqref{eq:mink-boundary-corrected} from both sides and applying Stokes'
		theorem, we find
		\begin{multline}
			\dfrac{n-1}{n}c\int_{\partial M}(\lambda f+\beta)\,da
			=c^2\int_{\partial M}fH\,da+\int_M\operatorname{div}\left(
			\dfrac{1}{2f}\nabla|\nabla f|^2
			-\dfrac{\Delta f}{nf}\nabla f
			\right)\,dv\\
			+\int_{\partial M}\dfrac{1}{f}
			[A(\nabla_{\partial M}f,\nabla_{\partial M}f)
			+c|\nabla_{\partial M}f|^2]\,da
			\label{eq:mink-before-divkato}
		\end{multline}
		where we used that \(\dfrac{\Delta f}{f}\dfrac{\partial f}{\partial\nu}
		=c(\lambda f+\beta).\)

		From Theorem~\ref{divkato}, we have
		\begin{align*}
			f\operatorname{div}\left(
			\dfrac{1}{2f}\nabla|\nabla f|^2
			-\dfrac{\Delta f}{nf}\nabla f
			\right)
			+\dfrac{1}{f}\mathfrak{L}_g^{\ast}(f)(\nabla f,\nabla f)
			\geq\dfrac{n-1}{n}f
			\left\langle
			\nabla\left(\dfrac{\Delta f}{f}\right),\nabla f
			\right\rangle.
		\end{align*}
		Since
		\begin{equation*}
			\nabla\left(\dfrac{\Delta f}{f}\right)
			=\nabla\left(-\lambda-\dfrac{\beta}{f}\right)
			=\dfrac{\beta}{f^2}\nabla f,
		\end{equation*}
		we obtain
		\begin{align*}
			\operatorname{div}\left(
			\dfrac{1}{2f}\nabla|\nabla f|^2
			-\dfrac{\Delta f}{nf}\nabla f
			\right)\geq\dfrac{1}{f^2}
			\left[-\mathfrak{L}_g^{\ast}(f)
			+\dfrac{(n-1)\beta}{n}g\right](\nabla f,\nabla f).
		\end{align*}
		Substituting this expression into \eqref{eq:mink-before-divkato}, we obtain
		\begin{multline*}
			\dfrac{n-1}{n}c\int_{\partial M}(\lambda f+\beta)\,da
			\geq c^2\int_{\partial M}fH\,da\\
			+\int_M\dfrac{1}{f^2}
			\left[-\mathfrak{L}_g^{\ast}(f)
			+\dfrac{(n-1)\beta}{n}g\right](\nabla f,\nabla f)\,dv\\
			+\int_{\partial M}\dfrac{1}{f}
			[A(\nabla_{\partial M}f,\nabla_{\partial M}f)
			+c|\nabla_{\partial M}f|^2]\,da.
		\end{multline*}
		Consequently,
		\begin{equation*}
			\dfrac{n-1}{n}
			\left(\lambda\int_{\partial M}f\,da
			+\beta\operatorname{Area}(\partial M)\right)
			\geq c\int_{\partial M}fH\,da.
		\end{equation*}

		Integrating $\Delta f=-\lambda f-\beta$ gives
		\begin{equation*}
			c\int_{\partial M}f\,da
			=\lambda\int_Mf\,dv+\beta\operatorname{Vol}(M).
		\end{equation*}
		Finally,
		\begin{multline*}
			\left(\int_{\partial M}f\,da\right)
			\left(\lambda\int_{\partial M}f\,da
			+\beta\operatorname{Area}(\partial M)\right)\\
			\geq n\left(\lambda\int_Mf\,dv
			+\beta\operatorname{Vol}(M)\right)
			\int_{\partial M}\dfrac{fH}{n-1}\,da.
		\end{multline*}

		Equality holds if and only if $\mathring{\nabla}^2f=0$. For the rigidity
		statement, assume in addition that $M$ and $\partial M$ are connected, and that
		$\partial M$ is convex with $\lambda=n$. Then,
		\begin{equation*}
			\nabla^2f=-\left(f+\dfrac{\beta}{n}\right)g.
		\end{equation*}
		The mixed component of this identity, together with
		$\partial_\nu f=-cf$, gives
		\begin{equation*}
			A(\nabla_{\partial M}f,\nabla_{\partial M}f)
			+c|\nabla_{\partial M}f|^2=0.
		\end{equation*}
		Since $A\geq0$ and $c>0$, it follows that
		$\nabla_{\partial M}f=0$. Hence, $f=\kappa>0$ on $\partial M$ for some constant
		$\kappa$. By assumption, $\kappa+\beta/n\neq0$. The tangential component of the Hessian equation gives
		\begin{equation*}
			A=\dfrac{\kappa+\beta/n}{c\kappa}g_{\partial M},
			\qquad
			H=(n-1)\dfrac{\kappa+\beta/n}{c\kappa}.
		\end{equation*}
		Since $A\geq0$ and $c\kappa>0$, we have $\kappa+\beta/n\geq0$. Thus, $\kappa+\beta/n>0$. Define
		\begin{equation*}
			\phi=f+\dfrac{\beta}{n},
			\qquad
			\tan\theta=\dfrac{\kappa+\beta/n}{c\kappa},
			\qquad
			\theta\in(0,\pi/2).
		\end{equation*}
		It follows that
		\begin{equation*}
			\nabla^2\phi+\phi g=0,
			\qquad\mbox{and}\qquad
			\tan\theta\,\partial_\nu\phi+\phi=0.
		\end{equation*}
		Moreover, $A=\tan\theta\,g_{\partial M}$ and $H=(n-1)\tan\theta$.
		Since $\partial_\nu\phi=-c\kappa<0$, the function $\phi$ is nonconstant. Therefore,
		all the assumptions of Theorem~\ref{thm:almaraz-barbosa-obata} are satisfied, and $(M,g)$
		is isometric to the standard spherical cap $\mathbb{S}_\theta^n$.

	\end{proof}

	\begin{proof}[Proof of Theorem~\ref{coroasymp}]
		The proof is based on Theorem~\ref{theoasymptotic}, since an electrostatic manifold is sub-static; see Proposition~\ref{properties}. We apply it with $E=0$ and $\partial M=\emptyset$. In fact, assume that $(M^{n},g,f)$ is a sub-static manifold satisfying $\Delta f=-\lambda f$, where we take $\lambda=-n$ for convenience. Then, Theorem~\ref{divkato} immediately gives
		\begin{equation*}
			\operatorname{div}\!\left( \dfrac1{2f}\nabla |\nabla f|^2 - \dfrac{\Delta f}{nf}\nabla f \right) \geq 0.
		\end{equation*}
		For $\rho>0$, let $M_\rho = M \cap \{r \leq \rho\}$. Applying the divergence theorem, we obtain
		\begin{equation*}
			0 \leq \int_{M_\rho} \operatorname{div}\!\left( \dfrac1{2f}\nabla |\nabla f|^2 - \dfrac{\Delta f}{nf}\nabla f \right)\,dv = \int_{S_\rho} \left\langle \dfrac1{2f}\nabla |\nabla f|^2 - \dfrac{\Delta f}{nf}\nabla f, \eta \right\rangle\,da,
		\end{equation*}
		where $\eta$ denotes the outward unit normal to $S_\rho$.

		Since $\Delta f=nf$, we have $\lambda=-n$ and $\beta=0$. Therefore, Lemma~\ref{lem:schwarz-vector} yields
		\begin{equation*}
			0 \leq \int_{S_\rho} \left\langle \dfrac1{2f}\nabla |\nabla f|^2 - \dfrac{\Delta f}{nf}\nabla f, \eta \right\rangle\,da = \omega_{n-1}\left[-(n-1)(n-2)\mathfrak M - \dfrac{(n-1)(n-2)^2\mathfrak M^2}{\rho^n}\right].
		\end{equation*}
		Taking $\rho\to+\infty$, we obtain
		\begin{equation*}
			0 \leq -\omega_{n-1}(n-1)(n-2)\mathfrak M,
		\end{equation*}
		which implies $\mathfrak M \leq 0$.

		Finally, since $\Delta f=nf$, the scalar curvature satisfies $R \geq -n(n-1)$. If $(M,g)$ is spin, the positive mass theorem for asymptotically hyperbolic manifolds (Theorem~\ref{PMT}) yields $\mathfrak M \geq 0$. Consequently, $\mathfrak M = 0$, and the rigidity statement of Theorem~\ref{PMT} implies that $(M,g)$ is isometric to the hyperbolic space.
	\end{proof}

	\begin{proof}[Proof of Theorem~\ref{ultmm}]
		By Theorem~\ref{divkato}, since $(M^n,g,f)$ is sub-static and $\Delta f=nf$, we have
		\begin{equation*}
			\operatorname{div} \left( \dfrac1{2f}\nabla|\nabla f|^2 - \dfrac{\Delta f}{nf}\nabla f \right) \geq \dfrac{n-1}{n} \left\langle \nabla\!\left(\dfrac{\Delta f}{f}\right), \nabla f \right\rangle.
		\end{equation*}
		Since $\dfrac{\Delta f}{f}=n$, the right-hand side vanishes, and therefore
		\begin{equation*}
			\operatorname{div} \left( \dfrac1{2f}\nabla|\nabla f|^2 - \dfrac{\Delta f}{nf}\nabla f \right) \geq 0.
		\end{equation*}

		For $\rho>0$, let $M_\rho=M\cap\{r\leq\rho\}$. Integrating over $M_\rho$ and applying the divergence theorem yield
		\begin{equation*}
			0 \leq \int_{\partial M} \left\langle \dfrac1{2f}\nabla|\nabla f|^2 - \dfrac{\Delta f}{nf}\nabla f, \nu \right\rangle\,da + \int_{S_\rho} \left\langle \dfrac1{2f}\nabla|\nabla f|^2 - \dfrac{\Delta f}{nf}\nabla f, \eta \right\rangle\,da,
		\end{equation*}
		where $\eta$ denotes the outward unit normal to $S_\rho$.

		Since $f=\kappa$ and $\nabla f=-c\nu$ on $\partial M$, we obtain
		\begin{equation*}
			\left\langle \dfrac1{2f}\nabla|\nabla f|^2 - \dfrac{\Delta f}{nf}\nabla f, \nu \right\rangle = -\dfrac1{c\kappa} \left( \dfrac12 \langle\nabla|\nabla f|^2,\nabla f\rangle - \dfrac{\Delta f}{n}|\nabla f|^2 \right).
		\end{equation*}
		Using the boundary decomposition
		\begin{equation*}
			\Delta f = \Delta_{\partial M}f + H\langle\nabla f,\nu\rangle + \nabla^2 f(\nu,\nu),
		\end{equation*}
		together with $\Delta_{\partial M}f=0$ and $\nabla f=-c\nu$, it follows that
		\begin{equation*}
			\dfrac12 \langle\nabla|\nabla f|^2,\nabla f\rangle = c^2\Delta f + c^3 H.
		\end{equation*}
		Hence,
		\begin{equation*}
			\left\langle \dfrac1{2f}\nabla|\nabla f|^2 - \dfrac{\Delta f}{nf}\nabla f, \nu \right\rangle = -\dfrac1{\kappa} \left( \dfrac{n-1}{n}c\Delta f + c^2 H \right).
		\end{equation*}
		Since $\Delta f=n\kappa$, we conclude that
		\begin{equation*}
			\int_{\partial M} \left\langle \dfrac1{2f}\nabla|\nabla f|^2 - \dfrac{\Delta f}{nf}\nabla f, \nu \right\rangle\,da = -(n-1)c\,\operatorname{Area}(\partial M) - \dfrac{c^2}{\kappa} \int_{\partial M}H\,da.
		\end{equation*}

		On the other hand, Lemma~\ref{lem:schwarz-vector}, applied with $\lambda=-n$ and $\beta=0$, gives
		\begin{equation*}
			\int_{S_\rho} \left\langle \dfrac1{2f}\nabla|\nabla f|^2 - \dfrac{\Delta f}{nf}\nabla f, \eta \right\rangle\,da = \omega_{n-1}\left[-(n-1)(n-2)\mathfrak M - \dfrac{(n-1)(n-2)^2\mathfrak M^2}{\rho^n}\right].
		\end{equation*}
		Passing to the limit as $\rho\to\infty$, we obtain
		\begin{equation*}
			0 \leq -(n-1)c\,\operatorname{Area}(\partial M) - \dfrac{c^2}{\kappa} \int_{\partial M}H\,da - (n-1)(n-2)\omega_{n-1}\mathfrak M.
		\end{equation*}
		Therefore,
		\begin{equation*}
			-(n-2)\kappa\omega_{n-1}\mathfrak M \geq c\,\kappa\,\operatorname{Area}(\partial M) + c^2 \int_{\partial M}\dfrac{H}{n-1}\,da.
		\end{equation*}

%%%%%%%%%%%%%%%%%%%%%%%%%%%%%%%%%%%%%%%%%%%%%%%%%%%%%%%%%%%%%%%%%%%%%%%%%%%%%%%%%%%%%%%%%%%%%%%%%%%%%%%%%%%%%%%%%%%%%%%%%%%%%%%%%%%%%%%%%%%%%%%%%%%%%%%%%%%%%%%%%%%%%%%%%%%%%%%

\iffalse
        Moreover,
        \begin{equation*}
            \int_{M}\Delta f dv = n\int_M fdv,
        \end{equation*}
        i.e.,
            \begin{equation*}
            \int_{\partial M}\langle\nabla f,\,\nu\rangle da + \lim_{\rho\to+\infty}\int_{S_\rho}\langle\nabla f,\,\eta\rangle da  = n\int_M fdv.
        \end{equation*}
        Thus, using that $\nu = -\dfrac{\nabla f}{\vert\nabla f\vert}$ we obtain
                    \begin{equation*}
           - c\operatorname{Area}(\partial M) + \lim_{\rho\to+\infty}\int_{S_\rho}\langle\nabla f,\,\eta\rangle da  = n\int_M fdv.
        \end{equation*}
        		Since
		\begin{equation*}
			\nabla f=g^{rr}f'\partial_r=f\left(\dfrac{n-2}{r^{n-1}}\mathfrak M-\dfrac{\lambda r}{n}\right)\partial_r\qquad\mbox{and}\qquad \eta = f\partial_r
		\end{equation*}
        we get
                            \begin{equation*}
           - c\operatorname{Area}(\partial M) + \omega_{n-1}\lim_{\rho\to+\infty}r^{n-1}\left(\dfrac{n-2}{r^{n-1}}\mathfrak M-\dfrac{\lambda r}{n}\right) = n\int_M fdv.
        \end{equation*}
\fi
        %%%%%%%%%%%%%%%%%%%%%%%%%%%%%%%%%%%%%%%%%%%%%%%%%%%%%%%%%%%%%%%%%%%%%%%%%%%%%%%%%%%%%%%%%%%%%%%%%%%%%%%%%%%%%%%%%%%%%%%%%%%%%%%%%%%%%%%%%%%%%%%%%%%%%%%%

		The mean-convexity of $\partial M$ makes the right-hand side nonnegative, forcing $\mathfrak M\leq0$. Since the sub-static manifold $(M^n,g,f)$ satisfies $\Delta f=nf$, its scalar curvature satisfies $R\geq-n(n-1)$. Assuming further that $(M^n,g)$ is spin (if $n\ge3$), the positive mass theorem, Theorem~\ref{PMT}, applies and yields $\mathfrak M=0$, thereby showing that $(M^n,g)$ is isometric to hyperbolic space.
	\end{proof}

    \

	\noindent\textbf{Conflict of Interest:} The authors declare no conflict of interest.

    \

	\noindent\textbf{Data Availability:} Not applicable.

    \

	\begin{acknowledgement}
		The authors thank Heitor Novais for suggestions that improved this work.
	\end{acknowledgement}

\end{document}